\documentclass[12pt]{amsart}
\usepackage{palatino}
\usepackage{amsfonts}
\usepackage{amssymb}
\usepackage{amscd}
\usepackage{amsmath}
\usepackage{amsfonts}
\usepackage{amsthm}
\usepackage{amssymb}
\usepackage{amscd}
\usepackage[all]{xy}
\usepackage{enumerate}
\usepackage{amsmath}
\usepackage{amsfonts}
\usepackage{amsthm}
\usepackage{amssymb}
\usepackage{amscd}
\usepackage[all]{xy}
\usepackage{enumerate}



\theoremstyle{plain}

\newtheorem{thm}{Theorem}[section]

\newtheorem{prop}[thm]{Proposition}

\newtheorem{cor}[thm]{Corollary}
\newtheorem{lem}[thm]{Lemma}

\theoremstyle{definition}
\newtheorem{defn}[thm]{Definition}

\newtheorem{nota}[thm]{Notation}

\newtheorem*{ackn}{Acknowledgment}

\theoremstyle{remark}
\newtheorem*{rem}{Remark}

\newcommand{\B}{\mathcal{B}}

\newcommand{\sE}{\mathcal{E}}
\newcommand{\E}{\mathcal{E}}
\newcommand{\sF}{\mathcal{F}}
\newcommand{\F}{\mathcal{F}}
\newcommand{\sG}{\mathcal{G}}
\newcommand{\sH}{\mathcal{H}}
\newcommand{\sI}{\mathcal{I}}
\newcommand{\I}{\mathcal{I}}

\newcommand{\sO}{\mathcal{O}}

\newcommand{\sT}{\mathcal{T}}

\newcommand{\sV}{\mathcal{V}}
\newcommand{\sW}{\mathcal{W}}

\newcommand{\mC}{\mathbb{C}}

\newcommand{\mP}{\mathbb{P}}
\newcommand{\mQ}{\mathbb{Q}}

\newcommand{\Qmath}{\mathbb{Q}}

\newcommand{\codim}{\mathrm{codim}\,}

\newcommand{\wdt}        {\widetilde}

\numberwithin{equation}{section}

\newcommand{\beba}  {\begin{equation}\begin{array}{rcl}}

\newcommand{\eaee}  {\end{array}\end{equation}}

\newcommand{\mfour} {{\overline {\mathcal M}}^{1}_{g,4}}

\newcommand{\mngon} {{\overline {\mathcal M}}^{1}_{g,n}}

\title{On the slope of fourgonal semistable fibrations}

\author{ Beorchia Valentina and Zucconi Francesco}

\address{Dipartimento di Matematica e Geoscienze,
               Universit\`a degli studi di Trieste\\
via Valerio 12/b, 34127 Trieste, Italy\\
\texttt{beorchia@units.it}}

\address{Dipartimento di Matematica e Informatica\\
via delle Scienze 206,
Universit\`a degli studi di Udine\\
Udine, 33100 Italy\\
\texttt{Francesco.Zucconi@uniud.it}}

\begin{document}

\maketitle

\markboth{Beorchia and Zucconi}{fourgonal fibrations}

\begin{abstract} We bound the slope of sweeping curves $B$ in the $4$-gonal locus $\mfour$. Our results follow from some Bogomolov-type inequalities for nef rank two vector bundles on ruled surfaces.
% As a byproduct of our techniques 

\end{abstract}
 \tableofcontents

%%%%%%%%%%%%%%%%%%%%%%%%%%%%%%%%%%%%%%%%%%%%%%%%%%%%%%%%%%%%%%%%%%%%%%%%%%
\section{Introduction}

In this paper we are concerned with semistable fibrations $f\colon S\to B$, that is flat surjective morphisms between a smooth surface $S$ and a smooth curve $B$ such that, for every 
$b\in B$, the fiber $F_b:=f^{-1}(b)$ is a semistable curve and if $b\in B$ is general then $F_b$ is smooth and of genus $g$. 

% If $F$ is smooth and $B$ passes through $[F]\in  {\mathcal M}_g$, 
% there exists a projective curve, which we still call $B$ for simplicity, and a morphism $B \to \mg$, which determines
%a semistable fibration $f\colon S\to B$. Now, 

Following Xiao \cite{X}, we can associate with $f\colon S \to B$ a rational number $s(f)$, called the slope of $f\colon S \to B$, and defined as follows:
$$
s(f): =\frac{K^2_{f}}{{\chi_f}}
$$
\noindent 
where $K_f=K_S -f^\star K_{B}$ is the relative canonical divisor, 
$\chi _f : = {\rm deg}\ f_\star  \omega _f $, and $\omega _f := {\sO}_S(K_f)$.

The starting point of our study is the celebrated Cornalba-Harris {\it{slope inequality}}; see: \cite {CH}, which asserts that if $f\colon S\to B$ is a family of semistable curves of genus $g$ over an integral complete curve and with smooth general fiber, then 
$s(f)\geq 4+\frac{4}{g}$. Moreover, if $s(f)=4+\frac{4}{g}$ and $g\geq 3$, then the modular image of $B$ is contained in the hyperelliptic locus, c.f. \cite[Theorem 8.4 page 391]{ACG}.
The same inequality for not necessarily semistable families has been proved by Xiao: \cite[Sec. 3 Theorem 2, p. 459]{X}.

A nice introduction of the main known results related to slope inequalities can be found in
\cite[page 438]{ACG}. Here we briefly recall that in \cite{S} Stankova 
%showed a very general statement on  curves $B\subset\mtrig$, where with $\overline  {\mathcal M}^{1}_{g,n}$ we denote the closure inside the moduli stack $\overline  {\mathcal M}_{g}$ of the locus of points $[F]$ corresponding to smooth curves $F$ with a $g^{1}_{n}$. In particular she 
showed that
 if $f\colon S\to B$ is a semistable fibration such that the general fiber is a smooth trigonal curve, then
$s(f)\ge {24(g-1)  \over 5g+1}$, and equality holds if and only if all fibers are irreducible, $S$ is a triple cover of a ruled surface $Y$ over $B$ and the ramification divisor, $R$, of the cover $\rho\colon S \to Y$ satisfies $R\equiv {1\over 3} \rho^\star \rho_\star R$. 
Under certain restrictive assumptions on the singular fibres, she also gave the better bound $s(f)\ge  5 - {6  \over g}$ if $g$ is even and the Maroni invariant of the general fiber is zero. 
Recently Barja-Stoppino \cite{BS}, and Fedorchuk-Jensen \cite{FJ} showed that if $g$ is even the bound
 $s(f)\ge  5 - {6  \over g}$ holds for zero Maroni invariant, without the hypothesis of semistability
 and with no assumptions on the singular fibers.
Finally, the trigonal odd genus case for sweeping families in the trigonal locus of the moduli space has been considered
in \cite{DP}.

\subsection{Our result}\label{ourresult}
\bigskip
Concerning the slope of fourgonal fibrations
%curves in $\overline {\mathcal M}^{1}_{g,4}$ 
the bounding problem is widely open. 
There are estimates on $s(f)$ by Barja-Zucconi \cite{BZ}) and by Cornalba-Stoppino \cite{CS}, \cite{St} for fibrations, which factorise through a double cover of a hyperelliptic fibration.
In this paper we establish some bounds on the slope of semistable fourgonal fibrations 
$f\colon S \to B$ with smooth general fiber, whose modular image is not contained in some specific loci. For instance, 
if $g$ is odd and the modular image in not contained in three divisors, then we prove that
$$
s(f)\ge  {16(g-1)\over 3g+1}.
$$
%if $g\ge 10 $ is odd, and $g \ge 10$ even, respectively,

%curve $B\subset \mfour$, where with $\overline  {\mathcal M}^{1}_{g,n}$ we denote the closure inside the moduli stack $\overline  {\mathcal M}_{g}$ of the locus of points $[F]$ corresponding to smooth curves $F$ with a $g^{1}_{n}$, with $B$ not contained
%in the boundary. 

For the definition of such loci and for a precise statement of our Main Theorem we need to describe our construction and to introduce some notation. 

Given $f\colon S \to B$, it is well known that, up to a finite base-change, $f$ factorises through a rational 
map $\rho\colon S\dashrightarrow Y$, where $\pi_{B}\colon 
Y\rightarrow B$ is a ruled surface; see Theorem \ref{fattorizzazione}. We resolve $\rho\colon S\dashrightarrow Y$
and we get a commutative diagram:
\begin{equation}
\label{eq:familyfamiglia}
\xymatrix{X \ar[d]_{\tau}
\ar[r]^{\pi}  &  Y \ar[d]^{\pi_{B}}\\
S \ar[r]^{f} & B,} 
\end{equation}\noindent
where $\pi \colon X \to Y$ is a {\it generically finite} morphism of degree $4$ and branched over a divisor $B(\pi)$ of $Y$.  
Now let $X\to{\widehat{X}}\to Y$ be the Stein factorisation of $\pi \colon X \to Y$ and let $\widehat\pi\colon\widehat X\to Y$ be the {\it{finite}} 
factorisation morphism. It turns out that if the modular image of $f$ is a curve $B 
\subset \mfour$, which does not interesect a certain proper subscheme $\Xi \subset \mfour$ (see Section 4.), the Stein factorisation of $\pi\colon X\to Y$ is a Gorenstein 
cover (see Definition \ref{gorensteincover}).
 Nevertheless, our theory allows to 
 estimate $s(f)$ under the weaker assumption that $\pi \colon X \to Y$ is finite over any
 point $p\in B(\pi)$ such that $p$ is {\it{not}} a simple node; and in this case we will say that $B\subset \mfour$
is a curve with {\it{good Gorenstein factorisation}}; see Definition \ref{defonozioneHM} and Subsection \ref{defonozioneHM}. We remark that there exist sweeping families satisfying such a condition, like for instance the Harris - Morrison families constructed in \cite[Theorem 2.5]{HMo}. 
 
With such an assumption, we find a partial resolution $\wdt\tau\colon \wdt X\to {\widehat {X}}$ 
and obtain the following diagram:

\begin{equation}
\label{eq:familyfamigliafamiglia}
\xymatrix{\wdt X \ar[d]_{\wdt\tau}
\ar[r]^{\wdt\pi}  &  \wdt Y \ar[d]^{\sigma\circ\pi_{B}}\\
{\widehat{X}} \ar[r]^{\pi_{B}\circ\widehat\pi} & B,} 
\end{equation}\noindent
where $\wdt\pi\colon\wdt X\to\wdt Y$ is a degree $4$ Gorenstein cover, $\sigma\colon \wdt Y\to Y$
 is the blow up of $Y$ at suitable points, and if we set $\wdt f := \pi_{B}\circ\widehat\pi\circ \wdt\tau$, then 
$\wdt f\colon\wdt X\to B$ is a (not necessarily semistable) fibration, whose slope can be bounded. Finally we show that $s(f)\geq s(\wdt f)$.

The bound on $s(\wdt f)$ relies on the theory of Gorenstein covers. Indeed, there is a projective bundle $\mP(\wdt\sE)$ 
over $\wdt Y$ and an embedding $\wdt j\colon\wdt X\to \mP(\wdt \sE)$,
which allows to express $s(\wdt f)$
in terms of the Chern classes of $\wdt \sE$ and 
of the rank two Casnati-Ekedahl {\it bundle of conics} 
$\wdt \sF :=\wdt\pi_\star \sI _{\wdt X, \mP(\wdt\sE)} (2)$.

%Note that $\wdt\pi\colon \wdt X\to \wdt Y$ gives a $g^1_4$ to the general fiber $F$ of $f\colon S\to B$. In particular the 
%restriction of $\sF$ to the general fiber $\wdt L$ of $\wdt Y\to B$, which we denote by $\sF_F$, is a splitted rank $2$ vector bundle over $\wdt L=\mP^1$ . 

Next we define the {\it Casnati-Ekedahl locus} ${\rm CE} (\mfour)$, which is the subscheme of $\mfour$ given by:
$$
{\rm CE} (\mfour):={\overline{ \{ [F]\in \mfour |\,\,  \sF_F \,\,{\rm{is\, not\, balanced}}\}}}
$$
\noindent
where a rank $r$ vector bundle $\sH$ on $\mP^1$ is called {\it balanced} if for the sequence of integers $(a_1,\dots,a_r)$ given by its splitting into the direct sums of line bundles, 
$\sH \cong \bigoplus _{i=1}^r \sO _{\mP^1} (a_i)$, it holds that:
$|a_i - a_j| \le 1$ where $i\neq j$ and $i,j=1,...,r$.

%Note that the definition of  ${\rm CE} (\mfour)$ makes sense if there exists a unique $g^1_4$ on $F$. 
If $g \ge 10$, the codimensions in $\mfour$ of the Casnati-Ekedahl loci are $1$ if $g$ is odd, and $2$ if $g$ is even, see Theorem \ref{brusac}.

Finally, we set $T\subset \mfour$ to be divisor of curves with a triple ramification, $D\subset \mfour$ the divisor of double simple ramification, and $\Upsilon \subset \mfour$ the subscheme corresponding to points of the boundary divisor $\delta\subset \mfour$ given by stable curves with a ramification in a singular point, or by stable curves with two or more nodes (see Definition \ref{luoghidaevitare}). We will say that the fibration satisfies the condition $(\dagger)$ if 

\bigskip
$(\dagger)$
$B\cap \Upsilon =\emptyset$ and if $Y$ has an irreducible negative section contained in the branch divisor, there are 
no triple nor total ramification points for $\pi:X\to Y$ over any of its points.

\bigskip

In the present paper we prove that:
\medskip

\noindent
{\bf{Main Theorem:}}{\it{ Let $f:S \to B$ be a
 semistable fourgonal fibration with good Gorenstein factorisation, and assume
 that the general fiber $F$ has genus $g \ge 10$. Denote again by $B\subset \mfour$ its modular image and assume that $B\not\subset T$ and $B\not\subset D$. Then:

\begin{enumerate}
\item
if $g$ is odd and $B\not\subset {\rm CE} (\mfour)$, then $s(f) \ge {16(g-1)\over 3g+1}$;

 \item
 if $g$ is even, $B\not\subset {\rm CE} (\mfour)$ and the condition $(\dagger)$ is satisfied, then $s(f) \ge {16(g-1)\over 3g+2} $;
 
 \item if the $4$-gonal morphism $h\colon F\to \mP^1$ does not factorise and the condition $(\dagger)$ is satisfied, then $s(f)\ge {24(g-1)\over(5g+3)},$
 with equality if and only if $f\colon S \to B$ factorises through a finite degree four cover $\pi\colon S\to Y$ of a ruled surface $Y\to B$,
whose ramification divisor $R$ satisfies $R \equiv {1\over 4} \pi ^\star \pi _\star R$;

\item
if $h\colon F\to \mP^1$ factorises 
 through a double cover of a hyperelliptic curve of genus $\gamma <{(g-3)\over 6} $ and the condition $(\dagger)$ is satisfied, then $s(f) \ge {4(g -1)\over (g-\gamma)}.$
 \end{enumerate}}}
\medskip

%We stress that if $B(\pi)$ admits only simple nodes, then 
%$f:S \to B$ is a fibration with good Gorenstein factorization and our Theorem applies
%regardless the locus where $\pi\colon X\rightarrow Y$ is not finite. 
%
%
%However, it seems that our condition on $f:S \to B$ can not be easily checked from
%some geometrical information on $B\subset\mfour$. 
%
%but once we consider the semistable fibration $f\colon S \to B$ induced by $B$ and the induced generically finite morphism $\pi \colon X \to Y$ as in diagram (\ref{eq:familyfamiglia}), it seems not very restrictive
% to study the case where the branch locus $B(\pi)$ is a nodal curve, and actually our theory is much more general than the one where $B(\pi)$ is a nodal curve. 

%Finally, we point out that if $g\geq 10$ and the (unique) gonal  cover of the general fiber $F\to\mP^1$ is simply ramified, then the arguments given in \cite[pages 343-347]{HMo} show that it passes a $1$- dimensional family with good Gorenstein factorisation through $[F]\in \mngon$. 

%We remark also that the main theorem confirms the general belief that fibrations with low $s(f)$  should occur when the fibers are special in their moduli space, as in the Cornalba-Harris Theorem are hyperelliptic fibrations. This is also the meaning to consider in the statement of the Stankova's conjecture a curve 
%$B\subset \mngon$ {\it{not contained in a certain codimension 1 closed subset}} of $\mngon$. The condition $g\geq 10$ is necessary since in the proof of theorem \ref{boundsc2f} we need that there exists a unique $g^1_4$ on the general fiber of $f\colon S\to B$. 
%
\begin{rem}
We remark that A. Patel, in his Ph. D. thesis, has proved that for 
$g\equiv 3 \mod 6$ the slope of a sweeping $4$-gonal family not contained entirely in the divisors $T$, $D$, ${\rm CE} (\mfour)$ and in the Maroni divisor 
$M(\mfour)$ (see Remark \ref{tesipatel}) is
bounded below by $s(f)\ge {11\over 2} - {15\over 2g}$. We prove in Remark \ref{tesipatel} that such a statement is consistent with our results.
\end{rem}

%As a corollary of our results we can recover some bounds in the case trigonal case for $g\ge 5$; see Theorem \ref{teorematrigonale} in the last section of this paper.

\subsection{The contents of each section}

In Section $2$, we recall some basic results of the theory of finite Gorenstein covers. For a fibration $f:S\to B$, which factors through a finite Gorenstein cover $\pi:S \to Y$ of a ruled surface over $B$, we express the slope of $f$ in terms of the $\pi$-relative canonical divisor and the Chern classes of the reduced direct image sheaf.

 We end Section $2$ with an important result (Theorem \ref{boundsc2}), establishing some 
 Bogomolov-type inequalities between the Chern classes of a rank two vector bundle on a ruled surface, under the assumption that the vector bundle is weakly positive 
 outside a zero-dimensional 
 subscheme.

In Section $3$ we introduce the Casnati-Ekedahl bundle of conics associated with a degree four cover. Using the
Viehweg Weak Positivity Theorem \cite[3.4]{V1} we show that both the reduced direct image sheaf and the bundle of conics are weakly positive outside the branch locus of the cover and a zero-dimensional subscheme. Theorem \ref{boundsc2} and some additional arguments allow then to conclude when $S$ is a Gorenstein cover of a ruled surface.

 In Section $4$ we define the property of having a good Gorenstein factorisation. Under such an assumption we can estimate the invariants of $f \colon X\to B$ in terms of those of a family $\wdt f\colon\wdt X\to B$ where $\wdt X$ is Gorenstein; see Theorem \ref{gorensteinfour}. Then in Theorem \ref{theoremsnfour} we show that the bounds claimed in our Main Theorem hold for $s(\wdt f)$. We end Section $4$ with the proof of the Main Theorem, which is now a consequence of the fact that $S$ is a minimal model of $\wdt X$.

%In Section 5 we collect some results on trigonal fibrations which follow by applying our method.
%In particular we find a bound for families of trigonal curves of odd genus with good Gorenstein factorisation.
%
%In the paper \cite{DP1} Patel and Deopurkar have found a bound for families of sweeping curves in the moduli space, which is slightly better than our one. 
%However their result depends on hypotheses, which are different from ours; see the brief discussion at the end of this paper.
%
%

%%%%RINGRAZIAMENTI%%%%%%%%%%%%%%%%%%%%

\ackn 
This research is supported by MIUR funds, 
PRIN project {\it Geometria delle variet\`a algebriche} (2015), 
coordinator A. Verra, and by GNSAGA of INdAM.

The first author is also supported by funds of the
Universit\`a degli Studi di Trieste - Finanziamento di Ateneo per
progetti di ricerca scientifica - FRA 2015.

The second author is also supported by funds of the
Universit\`a degli Studi di Udine - Finanziamento di Ateneo RICERCA LIBERA di ATENEO 2015.

The authors are grateful to M. {\'A}. Barja, L. Stoppino and E. Tenni for fruitful discussions had in September 2010, when they visited the Department of
Mathematics and Informatics in Trieste. The authors thank also P. Hawking,
M. Manetti, A. Verra, B. Fantechi and A. Alzati.

Finally, the authors thank the organizers of the Conference "New Trends in Algebraic Geometry", Universit\`a della Calabria, June 12-14 (2013), where the second author had the opportunity to present some of the results of this work.

We thank, also, the anonymous referee for the useful comments.

%%%%%%%%%%%%%RINGRAZIAMENTI FINE%%%%%%%%%%%%%%%%%%%

\section{Preliminary results}
In this paper a {\it fibration} $f\colon S \to B$ is a
flat proper surjective morphism, with smooth connected general fiber, from a surface with canonical singularities
 $S$ to a smooth curve $B$. Note that two dimensional canonical singularities are the same as du Val singularities, and they are analytically isomorphic to quotients of $\mathbb C^{2}$ by finite subgroups of $SL2(\mathbb C)$. In particular the canonical divisor is a Cartier divisor since they are rational Gorenstein singularities. We denote by $F_b$ the fiber over a point $b\in B$ and we denote by $g$ the genus of a general fiber. We need
to recall some results from the theory of Gorenstein covers.

Let $X$ and $Y$ be schemes. An affine morphism $\pi\colon X\to Y $
is called a {\it cover of degree $n$} if $\pi_\star\sO_X$ is a locally free sheaf of rank $n$; observe that
$\pi\colon X\to Y$ is a cover if and only if it is flat and finite. If $Y$ is smooth and $X$ is locally Cohen-Macaulay, then every finite surjective morphism is a cover.

There exists an exact sequence of the form $0\to \sO_Y\to \pi_\star \sO_X\to \sE^\vee\to 0$, where
$\sE ^\vee $ is a locally free $\sO_Y$ sheaf of rank $n-1$, called the {\it Tschirnahausen sheaf} of $\pi\colon X\to Y$. The above sequence splits and
$\pi_{\star}\sO_X = \sO_Y\oplus \sE^{\vee}$; see \cite{CE}. 

\subsection{Gorenstein Covers}
From now on we assume $Y$ to be a smooth surface. 

\begin{defn}\label{gorensteincover}
 A {\it Gorenstein cover} $\pi\colon S \to Y$ is a finite 
surjective morphism from a normal surface $S$ to a smooth surface $Y$
such that all $\pi$-fibers are Gorenstein schemes.
\end{defn}
The above definition is given in \cite{CE} in a more general set up. If the cover $\pi\colon S \to Y$
has Gorenstein fibers, then $S$ is Gorenstein and we have $(\pi_{\star}\sO_S)^\vee=\pi_\star\omega_{S/Y}$ \cite[exercise III 6.10]{Ha}, hence 
$\pi_\star\omega_{S/Y}=\sO_Y\oplus \sE$. We shall call $\sE$ the {\it reduced direct image} sheaf. We will indicate by $R$ the $\pi$-relative canonical divisor.

Finally, we shall denote by $\pi_Y\colon \mP(\sE)\to Y$ the
projective bundle associated with $\sE$. We recall the Casnati-Ekedahl Theorem:

\begin{thm}\label{Casnati}
Let $Y$ be a smooth surface and let $ \pi\colon S \to Y$ be a Gorenstein cover of degree $n\ge 3$. 
There exists a unique $\mP ^{n-2}$-bundle
$\pi _Y\colon \mP \to Y$ and an embedding $j\colon X \to \mP$ such that $\pi=\pi _Y \circ j$. Moreover
 $\mP\cong \mP (\sE)$ and $R$ satisfies:
$$
\sO _S (R) \cong j^\star \sO _{\mP (\sE)} (1).
$$
\end{thm}
\begin{proof} See \cite[Theorem 1.3]{CE}.
\end{proof}

For the rest of this and the next subsection we assume that 
$$
n \le 4
$$ 
and that
$S$ is a normal Gorenstein surface. In this case, as $\pi\colon S\to Y$ factorises through
the closed embedding $j\colon S\to \mP$ and the projective morphism $\pi _Y\colon \mP \to Y$, and
since a codimension $1$ or $2$ Gorenstein subscheme of a smooth variety is a
local complete intersection (see, for instance, \cite[21.10, p. 537]{E}), then $\pi\colon X\to Y$ is a local complete intersection (l.c.i. to short) morphism. Then
we can apply the
Grothendieck-Riemann-Roch theorem for singular varieties and proper l.c.i. morphisms 
(see: \cite[Corollary 18.3.1 (c), page 354]{Fu})
to write the invariants of $S$ in terms of the invariants of $Y$ and of the Chern classes of $\sE$. 

Let us denote by $A(Y)$ the Chow ring of $Y$ and by $\equiv$ the numerical equivalence.

\begin{lem}\label{scriviamolo}
Let $S$ be a normal Gorenstein surface and let $Y$ be a smooth surface. Let $ \pi\colon S \to Y$ be a finite morphism of degree $n$, 
and let $\sE$ be its reduced direct image sheaf.
Then in $A(Y) \otimes_\mathbb Z \Qmath$ we have:
\begin{enumerate}
\item \label{R} $\pi _  \star R \equiv 2 c_1 (\E)$,
\item \label{chi} $ \chi (\sO _S)= n \chi (\sO _Y) +{1\over 2} c_1 (\E) \cdot K_Y +{1\over 2} c_1 (\E)^2 - c_2 (\E)$;
\item \label{c12R}
$c_1 (\pi_\star \sO _S (2R))\equiv 3 c_1 (\E)$,

\item \label{c22R} $c_2 (\pi_\star \sO _S (2R)) = 4c_1 (\E) ^2 + c_2 (\E) -R^2$.
\end{enumerate}
\end{lem}
%%%%%%%%%%%%%%%%%%%%%%%%%%%%%%%%%%%%%%%%%%%%
\begin{proof} 
By the Grothendieck-Riemann-Roch theorem for $ \pi\colon S \to Y$ applied to the sheaf $\sO_S$ we can write:
${\rm ch} (\pi _! \sO _S) \cdot {\rm td} {\sT}_Y = \pi_  \star ({\rm ch }\sO _S \cdot {\rm td} {\sT}_S)$. In our case this means that $
{\rm ch} (\pi _  \star \sO _S) \cdot (1-{1\over 2}K_Y +\chi(\sO _Y))=\pi _  \star (1-{1\over 2} K_S +\chi(\sO_S))$, 
that is
$$
(n-c_1 (\E)+{1\over 2}(c_1 ^2 (\E) -2c_2 (\E))\cdot (1-{1\over 2}K_Y +\chi(\sO _Y))=\pi _\star (1-{1\over 2} K_S +\chi(\sO_S)).
$$
The divisorial part of the above equation in $A(Y) \otimes_{\mathbb Z} \Qmath$ gives: $-c_1 (\E)=-{1\over 2} (\pi _\star K_S - n K_Y)$
\noindent
and as $K_S \sim \pi ^\star K_Y + R$, where $\sim$ denotes the linear equivalence, we have $\pi _\star K_S \equiv nK_Y + \pi _\star R$ and $(1)$ follows. The equality between the codimension two cycles gives formula (\ref{chi}). Formulae (\ref{c12R}), (\ref{c22R}) follow by the same argument applied to the sheaf $\sO_S(2R)$ and by (\ref{R}) and (\ref{chi}).
\end{proof}
%%%%%%%%%%PRIMA FORMULA SULLO SLOPE%%%%%%%%%%%%%%%%%%%
\subsection{First formula for the slope}

We shall now express the slope of a fibration, which factorises through a Gorenstein cover $\pi$, in terms of the $\pi$-relative canonical divisor and the Chern classes of the reduced direct image sheaf.

Observe that since $S$ is Gorenstein, it admits a Cartier canonical divisor $K_{S}$. It follows that for the fibration $ f\colon S \to B$, there exists a Cartier relative canonical divisor
$$
K_{ f} := K_{S}- { f}^\star K_B.
$$ 
Furthermore, we set
$$
\chi_{f}:=\chi (f _\star \sO_{S} (K_{ f})).
$$
So the slope of $ f$ is defined as
$$
s(f):= {K_{f}^2 \over \chi_{f}},
$$
where the intersection number is taken in the sense of \cite{K}.
 
 \begin{rem} 
 If $S$ is a normal, Gorenstein surface with canonical singularities, then
\begin{equation}\label{k2chi}
K_f ^2=K_S ^2-8(g-1)(g(B)-1), \qquad \chi _f = \chi (\sO _S)-(g-1)(g(B)-1).
\end{equation}
The formulae are well known when $S$ is a smooth surface.
If $S$ is singular and $\tau: \wdt S \to S$ is a minimal resolution of the singularities, then
$$
\tau_\star \omega_{\wdt S} \cong \omega_S, \qquad K_{\wdt S}^2= K_S^2, \qquad \chi(\sO_{\wdt S}) =\chi(\sO_S),
$$
hence the smooth case formulae (\ref{k2chi}) still hold.

\end{rem}

\begin{defn}\label{vaalala}
We say that $f\colon S\to B$ factorises through the Gorenstein cover $\pi\colon S\to Y$ if there exists a fibration 
$\pi_B\colon Y\to B$ such that $f=\pi_B \circ\pi$. 
\end{defn}

\bigskip

In the rest of the paper, unless otherwise stated, $Y$ is a ruled surface and $\pi_B\colon Y\to B$ is the ruling morphism. 
It is well-known that the 
$\mathbb Q$-Neron-Severi group ${\rm{ NS}} (Y)_{\mathbb Q}:={\rm{NS}}(Y)\otimes_\mathbb Z\mathbb Q$
 is a $2$-dimensional vector space over $\mathbb Q$ and that ${\rm{ NS}} (Y)_{\mathbb Q}=[T_Y]\mathbb Q \oplus [L]\mathbb Q$
 where $[T_Y]$ is the numerical class of a section of $\pi_B\colon Y\to B$ and $[L]$ is the class of a ruling.

Let $T_0$ be the following $\Qmath$-divisor on $Y$:
$$
T_0 := T_Y - {1\over 2} T_Y ^2  L.
$$

Next lemma gives an expression for the first Chern class of the reduced direct image sheaf $\E$,
which will be particularly useful in the sequel.
%%%%%%%%%%%%%%%%%%%%%%%
\begin{lem}\label{c1} Let $f\colon S\to B$ be a genus-$g$ fibration which factorises through a degree $n$ Gorenstein cover $\pi\colon S\to Y$ such that 
$\pi_B\colon Y\to B$ is a ruled surface. Then 

$$
c_1 (\E) \equiv (g+n-1){T}_0+ \left(
{c_1 (\E)^2 \over 2(g+n-1)}\right)L.
$$
\end{lem}
%%%%%%%%%%%%%%%%%%%%%%%%%%%%%%%%%%%%%%%%%%%
\begin{proof}
Let $b\in B$ a general point. We know that ${\rm{ NS}} (Y)_{\mathbb Q}=[T_0]\mathbb Q \oplus [L]\mathbb Q$ where in this proof we set $L:=\pi_{B}^{-1}(b)$. We consider the restriction $\E_L$ of $\E$ to $L$. Then $c_1(\E)\equiv \deg c_1(\E _L){ T}_0 +\delta L$ for some $\delta \in \Qmath$. 

We first show that $\deg c_1(E_L)=g+n-1$.
Indeed, by definition $\pi^{*}L=F_{b}$. Hence by projection formula and by Lemma \ref{scriviamolo} (\ref{R}) it holds that
$\deg c_1(\E_L) = {1\over 2} R \cdot  F_b$. The fiber $F_b$ is general, hence it is transversal to $R$ and the ramification divisor of the induced morphism 
$\pi_{|F}\colon F\to L$ is given by $R_{|F_b}$. By the Riemann-Hurwitz formula it follows that $g+n-1={1\over 2} R \cdot  F_b=\deg c_1(\E_L)$.

Now, since $T_0 ^2 =0, \ L ^2 =0,\ T_0 \cdot L =1$, we get $c_1 (\E) ^2 = 2(g+n-1)\delta$ and the statement follows.
\end{proof}

Next proposition expresses the slope in terms of $R^2$ and the Chern classes of $\sE$.
 
\begin{prop}\label{slopegenerale} Let $f:S\to B$ be a fibration with $S$ normal and with canonical singularities, which factorises through a Gorenstein degree $n \le 4$ cover 
$\pi\colon S\to Y$ of a ruled surface $Y$.
Let $R$ be the $\pi$-relative canonical divisor and let $\E$ be the reduced direct image sheaf of $\pi\colon S\to Y$. Let 
$c_1 (\sE)$, $c_2 (\sE)$ be respectively the first Chern class and the second Chern class of $\E$.
Then
\begin{equation}\label{primaslope}
s(f) = { R^2 - {4\over g+n-1} c_1 (\sE)^2 \over {g+n-2 \over 2(g+n-1)} c_1 (\sE)^2 -c_2 (\sE)}.
\end{equation}
\end{prop}

\begin{proof} % We recall that $s(f)=\frac{K_f ^2}{\chi _f}.
%$
 By formula (\ref {k2chi}) we have
 $$
K_f ^2 = (R +\pi ^\star K_Y) ^2 -8(g-1)(g(B)-1)=
R^2 + 4 c_1 (\sE) \cdot K_Y +n K_Y^2 -8(g-1)(g(B)-1),
$$
where we applied Lemma \ref{scriviamolo} (\ref{R}) and the projection formula for
$\pi\colon S\to Y$.

 In the $\Qmath$-basis $T_0, L$ we have
$K_Y \equiv -2T_0 + (2g(B)-2)L$. Using Lemma \ref{c1} we obtain $c_1 (\sE)\cdot K_Y = -{c_1(\sE)^2 \over (g+n-1)} +2(g+n-1)(b-1)$, hence
$K_f ^2 = R^2 - {4\over (g+n-1)} c_1 (\sE)^2$. 

To show that $\chi _f =\frac{g+n-2}{2(g+n-1)}c_1 (\sE)^2 -c_2 (\sE)$, we use the formula 
given in (\ref{k2chi}) $\chi _f = \chi (\sO _S)-(g-1)(g(B)-1)$, and the equality given in Lemma \ref{scriviamolo} (\ref{chi}), taking into account the relation 
$\chi (\sO _Y)=1-g(B)$, which holds for any ruled surface.
\end{proof}

%\begin{cor}\label{upperc2}
%In the same hypotheses of Proposition \ref{slopegenerale} we have
%$$
%c_2(\sE) < \frac{g+n-2}{2(g+n-1)}c_1 (\sE)^2.
%$$
%\end{cor}
%
%\begin{proof}
%As $\chi _f = \chi (\sO _S)-(g-1)(g(B)-1)$ holds, this is a birational invariant, and for any smooth non isotrivial fibration the relative Euler characteristic is strictly positive.
%
%
%\end{proof}
%

\subsection{Weakly positive vector bundles on a ruled surface}\label{ventitre}

The aim of the next results is to find some suitable bounds on the invariants appearing in the formula for the slope given in
(\ref{primaslope}).

The crucial fact we shall use is that both the reduced direct image sheaf $\sE$ and the bundle of conics $\sF$ turn out to be weakly positive outside the branch locus of the cover and outside a zero dimensional subscheme. Let us recall the definition of weak positivity (see \cite[Definition 2.11, Remark 2.12.2]{V2}).

\begin{defn}
A locally free sheaf $\sG$ on a projective variety $Y$ is
{\it weakly positive} over $Y$ if for every ample invertible sheaf $\sH$ on $Y$ and for every $r >0$, the sheaf
$Sym^r \sG \otimes \sH$ is ample.
\end{defn}

We now consider the celebrated Viehweg's Weak Positivity Theorem \cite[3.4]{V1}, which states that 
the direct image of the relative canonical sheaf of a projective surjective morphism
between projective varieties is weakly positive on the complement of the branch locus. As $\sE$ is a quotient of $\pi_\star \omega_\pi$ and, as we shall see in Theorem \ref{casnati_discr}, $\sF$ coincides, outside a zero dimensional subscheme, with the reduced direct image sheaf associated with the discriminant morphism of the cover $\pi$, we will deduce that $\sE$ and $\sF$ are weakly positive outside the branch locus of the cover and outside a zero dimensional subscheme.

So we now study the Chern classes of weakly positive vector bundles on an open subscheme of a ruled surface, or more generally of a blow up of a ruled surface.

Let us fix the following: 
 
\begin{nota} \label{blowup}
Let $\sigma\colon\wdt Y \to Y$ be the blow-up of a ruled surface $\pi_{B}\colon Y\to B$ in a finite number of points
$q_1,..., q_s\in Y$. 
With $\wdt L$ we shall denote a fiber of $\wdt Y\to B$,
with
$T_{\wdt Y}$ the pull back of a section of $\pi_B:Y \to B$, and by
$E_i$, $i=1,..., s$ the exceptional divisors.

\end{nota}

\begin{prop}\label{c1wp}
Let $\sG$ be a vector bundle on $\wdt Y$, which is weakly positive outside a zero dimensional subscheme $\Theta \subset \wdt Y$.
%reduced direct image sheaf associated with a Gorenstein cover $\pi:X\to \wdt Y$.
%
%, and if its first Chern class satisfies
%$$
%c_1(\sE) \cdot \wdt L = c >0,
%$$
Then $c_1 (\sG)$ is nef.

%, and if the brach locus contains no curve
%of negative selfintersection, then
%$
%c_1(\sE)^2 \ge 0.
%$
%
\end{prop}

\begin{proof}
As $\sG$ is weakly positive outside $\Theta$, the line bundle $\det \sG=\sO_{\wdt Y} (c_1(\sG))$ satisfies the same property by
 \cite[Corollary 2.20]{V2}. By the definition of weak positivity, for any ample divisor $H$ on $\wdt Y$, $H \not\supset \Theta$, and for any integer $r>0$, the bundle
$$
\left ( Sym ^r \sO _{\wdt Y} (c_1(\sG) ) \right)
  \otimes \sO _{\wdt Y} (H)=
\sO _{\wdt Y} ( r c_1(\sG) +H)
$$ 
is ample on $\wdt Y\setminus \Theta$. Then for $m>>0$, the divisor $m(r c_1(\sG) +H)$ is very ample on $\wdt Y\setminus \Theta$, so the base locus of the linear system
$|m(r c_1(\sG) +H)|$
is at most zero-dimensional. Therefore for any effective divisor $Q$ on $\wdt Y$ we have that $m(r c_1(\sG) +H)\cdot Q\ge 0$, since otherwise $Q$
would be contained in the base locus. It follows that $r c_1(\sG)\cdot Q\ge 0$ and $c_1(\sG) \cdot Q \ge 0$, which shows the nefness of $c_1(\sG)$ on $\wdt Y$.
\end{proof}

The next result establishes some Bogomolov-type inequalities for rank two vector bundles on blows up of ruled surfaces, with the assumption that they are weakly positive outside a zero dimensional subscheme.

We need to recall the notion of {\it general splitting type}.
Let $\sG$ be a rank two vector bundle on a ruled surface $\pi_B\colon Y\to B$. A couple
$(\alpha,\beta)\in\mathbb Z\oplus\mathbb Z$, where $\alpha\leq \beta$, is said to be the 
{\it general splitting type} of $\sG$ if the restriction $\sG_L$ of $\sG$ to a general fiber $L$ of $\pi_B\colon Y\to B$ is isomorphic to $\sO_{\mP^1}(\alpha)\oplus\sO_{\mP^1}(\beta)$.
%; clearly it is possible that over a finite number of points $b_1,..., b_a$ of $B$ the splitting type $(\alpha_i,\beta_i)$ of $\sG_{L_i}$ is not equal to $(\alpha,\beta)$ where $L_i=\pi_B^{-1}(b_i)$, $i=1,..., a$. 

Observe that if $\wdt Y$ is the blow-up of $Y$ at a finite number of points, the general splitting type of a vector bundle $\sG$ on $\wdt Y$ with respect to the fibration $\wdt Y\to B$ can be defined similarly.

\begin{thm}\label{boundsc2} Let $\sigma\colon\wdt Y \to Y$ be the blow-up of a ruled surface $\pi_{B}\colon Y\to B$ at $q_1,..., q_s\in Y$. If $Y$ admits a negative section $T_Y$, assume that $q_i \not \in T_Y$ for any $i=1,\dots, s$.

 Let $\sG$ be a rank two vector bundle on $\wdt Y$,
with $(\alpha, \beta)$, $0< \alpha \le \beta$ the general splitting type of $\sG$. Let 
$c_1 (\sG)$, $c_2 (\sG)$ be respectively the first Chern class and the second Chern class of $\sG$, and set
$c_1 (\sG) \equiv (\alpha+\beta) T_{\wdt Y} + \delta \wdt L +\sum_{i=1}^s m_i E_i$,
where we use the notations of \ref{blowup}.
%$T_{\wdt Y} \subset \wdt Y$ is the pull back of a section of
% $\pi_B \colon Y \to B$, $\delta , m_i\in \mZ$, $\wdt L$ is a fiber of $\wdt Y\to B$
% and the $E_i$, $i=1,..., s$ are the exceptional divisors. 

Then
 \begin{enumerate}
 \item $c_2 (\sG) \ge {1\over 4} (c_1 (\sG)^2 + {\sum_{i=1}^s m_i^2})$ if $\alpha = \beta$;
 \item  \label{<} $c_2 (\sG) \ge  {\alpha\over 2 (\alpha +\beta)}  (c_1 (\sG) ^2 + {\sum_{i=1}^{s} m_i^2})$ if $\alpha < \beta$ and $\sG$ is 
 weakly positive outside a zero-dimensional subscheme $J\subset\wdt Y$.
 \end{enumerate}
 \end{thm}
\begin{proof} We recall that ${\rm{NS}}(\wdt Y)=
[T_{\wdt Y}]\mathbb Z\oplus [\wdt L]\mathbb Z\oplus_{i=1}^{s} [E_{i}] \mathbb Z$. We shall adapt the construction of Brosius \cite{B} to our more general case. 
We first show $(2)$, that is we assume $\alpha < \beta$ and that $\sG$ is weakly positive on $\wdt Y\setminus J$.
% We stress that $T_{\wdt Y} \subset \wdt Y$ is the total transform of a section of
 %$\pi_B \colon Y \to B$. 
 We set $\wdt\pi_{B}:=\pi_{B} \circ\sigma \colon\wdt Y\to B$.
 Since $\alpha<\beta$, $({\wdt \pi _{B}})_\star \sG (-\beta T_{\wdt Y})$ has rank one and it is a locally free sheaf on the curve $B$, since it is the direct image of a torsion free sheaf. Then by \cite [prop. III.9.8]{Ha} we can write:
$$
({\wdt \pi _{B}})_\star \sG (-\beta T_{\wdt Y})= \sO _B (N)
$$
where $N$ is a suitable divisor on $B$. By construction, 
 the natural map ${\wdt{\pi}}_{B}^\star \sO _B (N)\to  \sG (-\beta T_{\wdt Y})$
 is generically injective,
hence it is an injective map of locally free sheaves. It follows that the quotient sheaf is locally free outside a scheme 
$Z$ of codimension $2$. 
Denote by $D$ the first Chern class of such a quotient sheaf. Then the canonical extension of Brosius (\cite[Lemma 3, Proposition 2]{B}) of the vector bundle $\sG (-\beta T_{\wdt Y})$ has the form 
$$
0\to \sO _{\wdt Y}({\wdt{\pi}}_{B}^\star N) \to  \sG (-\beta T_{\wdt Y}) \to \sO _{\wdt Y} ((\alpha -\beta) T_{\wdt Y} + {\wdt{\pi}}_B^\star M+\sum_{i=1}^s m_i E_i) \otimes \I _Z \to 0,
$$
where ${\wdt{\pi}}_{B}^\star N + {\wdt{\pi}}_{B}^\star M \equiv \delta \wdt L$.

Twisting by $\beta T_{\wdt Y}$ we get
\begin{equation}\label{brosius_shift}
0\to \sO _{\wdt Y} (\beta T_{\wdt Y} + {\wdt{\pi}}_B^\star N) \to \sG \to \sO _{\wdt Y} (\alpha T_{\wdt Y} + {\wdt{\pi}}_B^\star M+\sum_{i=1}^s m_i E_i) \otimes \I _Z\to 0.
\end{equation}
Note that ${\wdt{\pi}}_B^\star N \cdot E_i =0$ for $i=1,..., s$. Hence we have:
$$
c_2 (\sG) = (\beta T_{\wdt Y} +{\wdt{ \pi}}_B^\star N) \cdot (\alpha T_{\wdt Y} + {\wdt{\pi}}_B^\star M +  \sum_{i=1}^s m_i E_i) )+{\rm{deg } Z} =
\alpha \beta T_{\wdt Y} ^2 + \alpha \deg (  \pi_B^\star N) +\beta \deg (  \pi_B^\star M)+{\rm{deg } Z}.
$$
By using the expression
$$
c_1 (\sG) \equiv ( \alpha + \beta) T_{\wdt Y} + {(c_1(\sG)^2 - (\alpha +\beta)^2T_{\wdt Y}^2 +\sum_{i=1}^{s}m_i^2)\over 2 (\alpha + \beta)}\wdt L  +\sum_{i=1}^s m_i E_i,
$$
we get $\deg (  {\wdt{\pi}}_B^\star N) = {(c_1(\sG)^2 - (\alpha +\beta)^2T_{\wdt Y}^2 +\sum_{i=1}^{s}m_i^2)\over 2 (\alpha + \beta)}
 - \deg (  {\wdt{\pi}}_B^\star M)$ and therefore:
 
 \begin{equation}\label{formulac2}
 c_2 (\sG) ={\alpha \over 2(\alpha +\beta)} c_1 ^2 +{\alpha \over 2} (\beta - \alpha) T_{\wdt Y} ^2 + (\beta - \alpha)\deg (  \pi_B^\star M)+{\rm{deg } Z} +\alpha\frac{\sum_{i=1}^s m_i^2}{2(\alpha+\beta)}.
 \end{equation}
Finally, observe that as $\sG$ is weakly positive on $\wdt Y\setminus J$, and
 $\sO _{\wdt Y} (\alpha T_{\wdt Y} + {\wdt{\pi}}_B^\star M + \sum_{i=1}^s m_i E_i)\otimes \I _Z$ is a 
 quotient line bundle of $\sG$ on $\wdt Y\setminus Z$, 
 we deduce that 
$\sO _{\wdt {Y}} (\alpha T_{\wdt Y} + {\wdt{\pi}}_B^\star M +\sum_{i=1}^s m_i E_i)$ 
is weakly positive on $\wdt Y\setminus (J\cup Z)$. 

By Proposition \ref{c1wp} we have that
$\alpha T_{\wdt Y} + {\wdt{\pi}}_B^\star M +\sum_{i=1}^s m_i E_i$ is nef on $\wdt Y$.

%and $m_i\leq 0$, $i=1,..., s$. 

%Indeed, by the definition of weak positivity, for any ample divisor $H$ on $\wdt Y$ and for any integer $r>0$, the bundle
%$$
%\left ( Sym ^r \sO _{\wdt Y} (\alpha T_{\wdt Y} + {\wdt{\pi}}_B^\star M + \sum_{i=1}^{s}  m_i E_i) ) \right)
%  \otimes \sO _{\wdt Y} (H)=
%\sO _{\wdt Y} ( r( \alpha T_{\wdt Y} + {\wdt{\pi}}_B^\star M +\sum_{i=1}^s m_i E_i) ) +H)
%$$ 
%is ample on $\wdt Y\setminus (J\cup Z)$. Then for $m>>0$, the divisor $m(r( \alpha T_{\wdt Y} + \pi_B^\star M +\sum_{i=1}^s m_i E_i) )+H)$ is very ample on $\wdt Y\setminus (J\cup Z)$. 
%Hence its base locus in $\wdt Y$ is at most zero-dimensional. 
%Now consider any effective divisor $Q\ge 0$ on $\wdt Y$. We observe that $m(r( \alpha T_{\wdt Y} + \pi_B^\star M +\sum_{i=1}^s m_i E_i) )+H)\cdot Q\ge 0 $, since otherwise $Q$
%would be contained in the base locus. Then $( \alpha T_{\wdt Y} + \pi_B^\star M +\sum_{i=1}^s m_i E_i) ) \cdot Q\ge 0$.
In particular, if we choose an irreducible $T_{\wdt Y}$ so that $T_{\wdt Y} \cdot E_i=0$,
we have $(\alpha T_{\wdt Y} + \pi_B^\star M) \cdot T_{\wdt Y} \ge 0$, that is
\begin{equation}\label{m}
\deg ( \pi_B^\star M)  \ge -\alpha T_{\wdt Y} ^2.
\end{equation}
If there exists a section $T_{\wdt Y}$ such that $T_{\wdt Y} ^2 = 0$ or $T_{\wdt Y} ^2 < 0$, by the hypothesis $q_i\not \in T_Y$ we can choose such a section to get the bound (\ref{m}). We get the statement (2) in this case from formula (\ref{formulac2}), as $Z$ is an effective $0$-dimensional cycle.

Assume now that for any section $T_Y$ of $Y$, its $\sigma$-pull-back 
$T_{\wdt Y}$ satisfies $T_{\wdt Y}^2 >0$, which is equivalent to saying that 
 $Y=\mP (\sV)$ with $\sV$ a stable rank two vector bundle on the curve $B$. Take $T_{ Y}$ to be the tautological divisor of $\mP (\sV)$. By Miyaoka Theorem \cite{Mi} the normalised tautological divisor $T_{ Y} - {1\over 2} \pi_B^\star c_1 (\sV)\equiv T_{ Y} -{1\over 2} T_{ Y}^2 L$ is a nef $\mQ$-divisor. Then for any rational number $\epsilon >0$ we have that
$$
T_{Y} - {1\over 2} T_{Y}^2 L +\epsilon T_{Y}
$$
is $\mQ$-ample. In particular we have that for $l>>0$
\begin{equation}\label{modestia}
l(T_{\wdt Y} - {1\over 2} T_{\wdt Y}^2 \wdt L +\epsilon T_{\wdt Y})\cdot (\alpha T_{\wdt Y} + \pi_B^\star M)\ge 0,
\end{equation}
since we can avoid the points $q_1,.., q_s$.

By equation (\ref{modestia}) we have that $(\alpha/2 +\epsilon)T_{\wdt Y}^2 +(1+\epsilon)\deg ( \pi_B^\star M) \ge 0$. Since this holds for any $\epsilon >0$, we have
$\deg ( \pi_B^\star M) \ge -\alpha/2  \  T_{\wdt Y}^2$. Using such inequality in equation (\ref{formulac2}) we conclude.

Now we show $(1)$. We have $\alpha = \beta$ and the claim can be proved in the same way as above, by taking 
into account that ${\wdt \pi}_{B}^\star (\wdt\pi_{B})_{\star} \sG (-\beta T_{\wdt Y})=:\sH$
has rank two and its Bogomolov discriminant $\Delta(\sH)= 4c_2 (\sH)-c_1(\sH)^2$
is zero, since $\sH$ is the pull back of a vector bundle on a curve.
On the other hand,
the exact sequence
$$
0\to \sH \to \sG (-\beta T_{\wdt Y}) \to \sI_Z\otimes_{\sO_{\wdt Y}}\sO_{\wdt {Y}} \to 0
$$
gives $Z=c_2 (\sG(-\beta T_{\wdt Y})) = c_2 (\sG) + c_1 (\sG) \cdot (-\beta T_{\wdt Y}) + \beta ^2 T_{\wdt Y}^2$;
as $Z$ is effective, we have $c_2 (\sG(-\beta T_{\wdt Y}))\ge 0$ and by computing explicitly the last expression, the claimed bound follows.
\end{proof}

%In our study of $1$-dimensional families of fourgonal curves we will use often the famous concept of weak positivity introduced by Viehweg, as well as Viehweg's Theorem \cite[3.4]{V1} to obtain, via Theorem \ref{boundsc2}, the desired bound on the slope.
% 

\section{Slope and Chern classes of the the Casnati-Ekedahl bundle of conics}
In this section we will show the claims of the Main Theorem for a semistable fibration $f\colon S\to B$ which 
factorises through a degree $4$ Gorenstein cover $\pi\colon S\to Y$ 
and such that its general fiber $F$ is a smooth fourgonal curve. 

\subsection{The Casnati-Ekedahl bundle of conics}

If $\pi\colon S\to Y$ is a degree $4$ Gorenstein cover by Theorem \ref{Casnati} there exists a unique $\mP ^{2}$-bundle
$\pi _Y\colon \mP \to Y$ and an embedding $j\colon S \to \mP$ such that $\pi=\pi _Y \circ j$,  
$\mP\cong \mP (\sE)$ where $\sO _S (R) \cong j^\star \sO _{\mP (\sE)} (1)$. 
Moreover, by \cite[Proof of Step B); p. 445]{CE}, the push forward of the exact sequence 
$$
0\to \I _{S, \mP (\sE)}(2) \to \sO _{\mP (\sE)} (2) \to \sO _S (2) \to 0
$$
gives the following exact sequence of locally free sheaves on the surface $Y$:
\begin{equation}\label{concon}
0 \to \pi_\star \I _{S, \mP (\sE)}(2) \to Sym ^2 (\sE) \to \pi_\star \sO _S (2R) \to 0.
\end{equation}

The sheaf $\F:=\pi_\star \I _{S, \mP (\sE)}(2)$ is called {\it{Casnati-Ekedahl bundle of conics}}. 
The next proposition shows how to write $R^2$ in terms of the classes 
$c_1(\sE)^2$, $c_2 (\sE)$, $c_2 (\sF)$. 

%%%%%%%%%%PROPOSIZIONEINVARANTI%%%%%%%%%%%%%
\begin{prop}\label{invarianti} Let $f\colon S\to B$ be a genus $g$ fibration which 
factorises through a degree $4$ Gorenstein cover $\pi\colon S\to Y$ where
 $Y$ is a smooth surface. Then
 \begin{equation}
 c_1(\sF)=c_1(\sE),
 \end{equation}
 and
\begin{equation}\label{formulaR2}
R^2= 2c_1 (\E)^2 -4c_2 (\E) +c_2 (\F).
\end{equation}
\end{prop}
\begin{proof}
By sequence (\ref{concon}) and by a standard computation we have:
\begin{equation}\label{csym}
 c_1 (Sym ^2 (\sE))=4 c_1 (\sE), \quad c_2 (Sym ^2 (\sE))=5 c_2 (\sE)+5
 c_1 (\sE)^2.
 \end{equation}
\noindent
By Lemma \ref{scriviamolo} (\ref{c12R}) we have $c_1 (\pi_\star \sO _S (2R))\equiv 3 c_1 (\E)$ and by Lemma \ref{scriviamolo} (\ref{c22R})
$$
c_2 (\pi_\star \sO _S (2R)) =c_2 (\E) + 4c_1 (\E) ^2-R^2.
$$
 Hence from the sequence (\ref{concon}) we get $c_1 (\F) =  c_1 (\E)$ and
$ c_2 (\F)= R^2 +4 c_2 (\E) -2 c_1 (\E) ^2$, which gives our claim.
\end{proof}

\begin{cor}\label{invariantiinvarianti}
Let $f\colon S\to B$ be a genus $g$ fibration which 
factorises through a degree $4$ Gorenstein cover $\pi\colon S\to Y$ where 
$\pi_B\colon Y\to B$ is a ruled surface. Then
\begin{equation}\label{slope4}
s(f) ={ {2(g+1)\over (g+3)} c_1(\sE)^2 - 4c_2 (\sE)+c_2 (\sF)  \over {(g+2)\over 2(g+3)} c_1 (\sE)^2 - c_2 (\sE)}.
\end{equation}
\end{cor}
\begin{proof}
It follows from Proposition \ref{invarianti} and from Proposition \ref{slopegenerale}.
\end{proof}

%The next step is to estimate $c_2 (\F)$ in terms of $c_1(\sF)^2=c_1(\E)^{2}$.
%We shall give an upper bound on $R^2$, which will follow from the Algebraic Index Theorem. 
%But we first need to show that $R^2 \ge 0$. By
%
 
The next step consists in bounding the second Chern class of the bundle of conics $\sF$. We shall actually apply Theorem \ref{boundsc2}, with an additional argument. Indeed, the results \cite[Proposition 4.4, remarks on page 1364]{C} allow to prove that
$\sF$ is weakly positive outside the branch locus $B(\pi)$ and a zero-dimensional subscheme. 

The formulation of the statement requires some preliminaries.

\begin{defn}\label{plan}
Let $\pi\colon S \to Y$ be a Gorenstein degree four cover, with $Y$
a smooth surface. We say that $y\in Y$ is {\it$\pi$-planar} 
(see \cite[Definition 3.2]{C})) if the fiber scheme $\pi^{-1} (y)$ is isomorphic to the scheme
$Spec \left( {k(y) [u,v] \over (u^2, v^2)} \right)$.
\end{defn}

The local analysis given in \cite[remarks on page 1364]{C} yield the following:

\begin{lem}\label{planar}
If $S$ is smooth, a point $y\in Y$ is $\pi$-planar iff $y$ is at least a fourfold point of the branch divisor $B(\pi)$ of $\pi\colon S \to Y$.
\end{lem}

Let $\pi\colon S \to Y$, where $Y$ is a smooth surface,
be a Gorenstein degree four cover with reduced direct image sheaf $\sE$ and Casnati-Ekedahl bundle of conics
$\sF$. Then the {\it discriminant scheme} $\Delta (S)$ and {\it discriminant map} $\Delta (\pi):\Delta (S) \to Y$ are defined \cite[Definition 4.1]{C}. The discriminant scheme corresponds to the locus of degenerate conics
in the pencils of conics associated with the fibers of $\pi$.

%map
%$$
%\Delta: H^0 (\sF^\vee \otimes Sym^2 \sE)\to H^0(Sym ^3\sF \otimes ({\rm det} \ \sF)^{-1})
%$$
%\end{defn}
%The bundle of conics $\sF$ turns out to be weakly positive 
%outside the branch locus of $\pi$ and the scheme of $\pi$-planar points. More precisely:

\begin{thm}\label{casnati_discr} 
The discriminant map $\Delta (\pi) : \Delta (S) \to Y$ associated with a degree four Gorenstein cover $\pi:S \to Y$
 is a generically finite morphism of degree three. Moreover,
$\Delta (\pi)$ is not finite over a point $y\in Y$ iff $y\in Y$ is $\pi$-planar, and in this case $\Delta (\pi) ^{-1} (y) \cong \mP ^1 _{k(y)}$. 
%Moreover, if also $S$ is smooth, a point $y\in Y$ is $\pi$-planar iff $y$ is at least a fourfold point of the branch divisor $B(\pi)$ of $\pi\colon S \to Y$.

Finally, if there are no $\pi$-planar points, then 
$\Delta (\pi) : \Delta (S) \to Y$ is a Gorenstein cover of degree three, the corresponding reduced direct image sheaf is $\sF={{(\Delta (\pi))_\star}  \omega _{\Delta (S)/ Y}}/{ \sO _Y}$ and the branch divisors of $\pi$ and $\Delta(\pi)$ coincide.
\end{thm}
\begin{proof} See \cite[Proposition 4.4]{C}, where a stronger version of the statement is shown.\end{proof}

As a consequence we have the following:

\begin{prop}\label{debolmentepospos}
Let $S$ be a normal surface and $Y$ a smooth surface. If $\pi\colon S \to Y$ is a Gorenstein cover of degree four with reduced branch divisor $B(\pi)$,
then the Casnati-Ekedahl bundle of conics $\sF$ is weakly positive outside $B(\pi)$ and a zero-dimensional scheme $J$.
\end{prop}
\begin{proof} By Theorem \ref{casnati_discr}, $y\in Y$ is a $\pi$-planar point if $\pi^{-1}(y)$ is supported on a singular point of $S$ or it is at least a fourfold point of $B(\pi)$. 
Since $S$ is normal, $Y$ is smooth and $B(\pi)$ is reduced, there is at most a finite number of $\pi$-planar points by Lemma \ref{planar}, hence the discriminant morphism $\Delta (\pi) : \Delta (S) \to Y$ is Gorenstein and finite outside a zero-dimensional subscheme $J$, supported on the $\pi$-planar points, by Theorem \ref{casnati_discr}.

Moreover, we observe that, by a similar argument to the one in \cite[Proposition 4.5 i), Corollary 4.11] {C}, one can deduce that $\Delta (S)$ is integral since $S$ is normal. 

 Therefore $\sF$ concides with the $\Delta(\pi)$ reduced direct image outside $J$, and by the Viehweg Weak Positivity Theorem \cite[3.4]{V1} $\sF$ is weakly positive outside $B(\pi)\cup J$.
 \end{proof}

Next we need to determine the general splitting type of $\sF$. This will be done using Schreyer's results on fourgonal curves.

\subsection{Chern classes of the Casnati-Ekedahl bundle of conics}

%By Maroni Theorem c.f. \cite[Theorem 5]{DZ}, 
%we know that if $F$ is a curve of genus $g$ then $F$ is $n$-gonal if and only its canonical model $C\subset\mP^{g-1}$ is contained 
%in a rational normal scroll of dimension $n-1$ called $n$-gonal scroll of $F$. 
%Moreover the fibres of the scroll cut on $C$ precisely the $g^1_n$.
% From now on we will not always distinguish between a curve $F$ of genus $g$ and its canonical model $C\subset\mP^{g-1}$. 
%If $F$ is a $n$-gonal curve and $h\colon F\to\mP^1$ is the associated morphism, we can define the {\it{reduced image}}
%$\sH_F:=\frac{h_\star\omega_h}{\sO_{\mP^1}}$ and 
%$\sH_F$ is locally free of rank $n-1$.

%In the case of a fourgonal cuve $F$, that is $n=4$, we have that the gonal scroll is of dimension $3$, and we denote it by $W$. 
%We also denote by $H$ the hyperplane section of $W$. Note that $H$ is given by the tautological divisor  of the not immersed model $W'$ of $W$. Let
% $\pi_{W'} \colon W' \to \mP^1$ be the natural projection and denote by $\Pi$ the $\pi_{W'}$-fiber over a general point $t\in \mP^1$. 
%We identify $\Pi$ to its image in $\mP^{g-1}$ by the tautological morphism $\phi_{|H|}\colon W'\to W\subset \mP^{g-1}$.
% In this subsection we denote by $\pi_{F}\colon F\to\mP^1$ the $4$-to-$1$ covering induced by the $g^1_4$. 

Let $F$ be a fourgonal curve of genus $g \ge 10$. By the Geometric Riemann-Roch Theorem, the 
span of any gonal divisor on the canonical model of $F$ is two dimensional, and the union of such spans determine the three-dimensional gonal scroll $W \subset \mP^{g-1}$, containing the canonical model of $F$ (c.f. \cite[Theorem 5]{DZ}). Set $H$ to be the hyperplane divisor on $W$ and let $\Pi$ be a fiber of the natural projection $W\to \mP^1$. With these notations we have:

\begin{thm}\label{Schreyer} There exist $b_1, b_2\in\mathbb N$ such that 
the canonical model $C$ of a fourgonal curve $F$ of genus $g \ge 5$ is the complete
intersection $C=Q_1 \cap Q_2$, where $Q_{1}\in |2H -b_1 \Pi|$ and $Q_2 \in  |2H -b_2  \Pi|$.
The two integers $b_1$ and $b_2$ satisfy the following relations:
$$
0\le b_2 \le b_1 \le g-5, \  \ \  b_1 + b_2 = g-5.
$$
Moreover $\pi_{F}\colon F\to \mP ^1$ factorises through a
double cover of a curve of genus $\gamma <{(g-3)\over 6} $ if and only if $ b_1 > {2\over 3} (g-3)$ and $\gamma = {b_2 \over 2}+1$.
\end{thm}
\begin{proof} See \cite[Sections 6.2, 6.3, 6.4, 6.5, 6.6]{Sch}.
\end{proof}

\begin{cor}\label{gensplit}
The generic splitting types $(\alpha,\beta)$ of the Casnati-Ekedahl bundle for a fourgonal family satisfy the following:
$$
4 \le \alpha \le \beta \le g-1, \qquad \alpha + \beta = g+3.
$$
Moreover, $\pi_{F}\colon F\to \mP ^1$ factorises through a
double cover of a curve of genus $\gamma <{(g-3)\over 6} $ if and only if $ \beta> {2g+3\over 3} $ and $\gamma = {\alpha \over 2}-1$.

\end{cor}

\begin{proof}
%In this proof we use the notation $F$ for the general fiber of $f\colon S\to B$.
% Let $(\alpha, \beta)$, $\alpha\le \beta$ be the splitting type of $\sF$ on the general fiber $L$ of $Y$. 
%
%First we need to bound $\alpha$. Consider $F\subset S$ and 

Let $L$ be a general fiber of $Y$,
set $V= \mP (\E \otimes \sO _{L})$ and let $T_V$ be the tautological divisor of $V$.
% By definition we have $\F \otimes \sO _{L} =\pi_\star \I _{S,\mP} (2 T_{\mP})\otimes \sO _{L}$. 
By projection formula $\pi_\star \I _{S,\mP} (2 T_{\mP}) \otimes \sO _{L} \cong\pi_\star ( \I _{S,\mP} (2 T_{\mP}) \otimes \pi^\star\sO _{L})$, so
$
\F \otimes \sO _{L} =
\pi_\star (\I_{F, V} (2 {T_{\mP}}_{|V})).
$
%since $\I_{F, V} (2 {T_{\mP}})_{|V}=\I _{S,\mP} (2 T_{\mP}) 
%\otimes \pi_{b}^{\star}\sO _{L}$. 

We can determine the minimal free resolution of 
$\I_{F, V}$ using Theorem \ref{Schreyer}.

Indeed, let $W=\mP (\sW)\subset\mP^{g-1}$ be the scroll 
containing the canonical model $C$ of $F$, and let $T_W$ be the tautological divisor
on $W$.

%Let $W'=\mP (\sW)$ be the abstract model of $W$.
% In particular $\sW$ is a rank $3$ vector bundle over $\mP^1$. 
By Theorem \ref{Schreyer} 
a minimal free resolution of $\I _{C, W} (2T_W)$ has the following form:
$$
0\to \sO _{W} (-2T_W +(g-5)\Pi) \to  \sO _{W} (b_1 \Pi)
\oplus \sO _{W}(b_2 \Pi) \to \I _{C, W} (2 T_W) \to 0.
$$
\noindent
%We have assumed that $g \ge 10$ then by \cite{AC} the $g^1_4$ on $F$ is unique.
%Then it is unique the two dimensional scroll containing $C$ as a four-secant 
%curve. By Theorem \ref{Casnati} $S$ embeds in $\mP$ via $|R|$ then 
As the pull back of $T_V$ to $F$ is the ramification divisor, we have
$$
\E \otimes \sO _{L}  \cong \sW (- K_{\mP ^1}).
$$
Now consider the isomorphism between the two scrolls: $\phi: V \to W\subset\mP^{g-1}$. 
By construction it follows that $H \sim \phi^\star (T)- 2\Pi$. Then a minimal free resolution of $\I_{F, V} (2 T_V)$ is given by:
$$
0\to \sO _{V} (-2T_V+(g+3)\Pi) \to  \sO _{V} ((b_1+4) \Pi) \oplus \sO _{V}((b_2+4) \Pi) \to \I_{F, V} (2 T_V)\to 0.
$$
Since ${\pi_{|V}}_\star \sO _{V} (-2T_V +(g+3)\Pi)= R^1 {\pi_{|V}}_\star \sO _{V} (-2T_V +(g+3)\Pi)=0$, we have
$$
{\pi_{|V}}_{\star}  \I_{F, V} (2 T)\cong \sO _{\mP ^1} (b_1 +4) \oplus \sO _{\mP ^1} (b_2 +4) ,
$$
so we get the claim by setting 
$b_2 =\alpha-4$, $ b_1 =\beta - 4$ in Theorem \ref{Schreyer}.
\end{proof}

%We recall that the claim of the Stankova's conjecture is about $1$-dimensional families $B\subset \mfour$ not contained in a certain codimension 1 closed subset of $\mngon$. 
%In our case we consider curves definitely not contained in some subloci of $\mfour$ that we are going to describe. 
%We need to recall some notions. 

%A rank $r$ vector bundle $\sH$ on $\mP^1$
%is called {\it balanced} if for the sequence of integers $(a_1,\dots,a_r)$ given by its splitting into the direct sums of line bundles, 
%$\sH \cong \bigoplus _{i=1}^r \sO _{\mP^1} (a_i)$, it holds that:
%$|a_i - a_j| \le 1$ where $i\neq j$ and $i,j=1,...,r$.

We shall see that if we consider curves $B\subset \mfour$ not contained in some specific closed subschemes, then the bundle of conics is in fact balanced, so we will get some better bounds on the general splitting type. To this purpose let us introduce the following locus:

%The closure in $\mfour$ of the loci where $\sH_F$ is not balanced is called {\it the Maroni locus}: 
%
%$$
%{\rm Mar} (\mfour):={\overline{\{ [F]\in \mfour |\,\,  \sH_F \,\,{\rm{is\, not\, balanced}}\}  }}.
%$$\noindent
%and analogousely the 

\begin{defn} The
{\it Casnati-Ekedahl locus} ${\rm CE} (\mfour)$ is the closure of the locus in $\mfour$ corresponding to curves with non balanced bundle of conics:

$$
{\rm CE} (\mfour):={\overline{ \{ [F]\in \mfour |\,\,  \sF_F \,\,{\rm{is\, not\, balanced}}\}}}
$$

\end{defn}
% \begin{rem}
% Our definition is sligthly different from the one given in \cite[Definition 2.1]{DP2}.
% \end{rem}
Such a locus turns out to be a proper closed subscheme by the following result.

\begin{thm} \label{brusac} 
If $g \ge 10$, the codimension in $\mfour$ of the 
%Maroni and 
Casnati-Ekedahl locus is given by:

$$
\codim {\rm CE} (\mfour)=
\left\{
\begin{array}{ll}
1 & {\rm if}\quad g \ {\rm odd},\\
2 & {\rm if}\quad g \ {\rm even}\\
\end{array}
\right .
$$
%$$
%\codim {\rm Mar} (\mfour) =
%\left\{
%\begin{array}{ll}
% 1 & {\rm if}\quad g \equiv 0 \mod 3\\
%2 & {\rm if}\quad g \equiv 1, 2 \mod 3.\\
%\end{array}
%\right .
%$$
\end{thm}
\begin{proof}
The dimensions of such subschemes can be computed using the results in \cite{BSa}.
Since the authors use different notation and define some loci, which are in some cases slightly different 
from the one we are considering, we briefly sketch the computation.

The authors introduce the invariant $\lambda$, which is in general the minimum degree of a linear series distinct from the $g^1_4$ \cite[Theorem 6.10]{BSa}. 
 We observe that 
the locus ${\rm CE} (\mfour)$ is equal, in the authors' notation, to the stratum $\overline {\mathcal M}^{\lceil {g\over 2}\rceil} _g$, given by the closure of the curves satisfying $\lambda \le \lceil {g\over 2}\rceil$.

Indeed, 
the relation between $\lambda$ and $\beta$, where $(g+3-\beta, \beta)$, $\beta \ge (g+3)/2$,
 is the splitting type
of the bundle of conics of a given fourgonal curve, can be obtained by the formula 
$$
{\rm deg}\ Z = g+\lambda -5
$$
given in 
\cite[Theorem 4.4]{BSa} for $t=0$, which expresses the minimum degree of a surface $Z$ ruled by conics
containing the canonical model of the curve. By Schreyer's' Theorem \ref{Schreyer}, the class of such a surface in the gonal scroll
is given by $2H - b_1\Pi$, where $H$ is the tautological divisor of the gonal scroll $\sE (-2)$. Since $b_1 = \beta - 4$ (see the proof of Corollary
\ref{gensplit}), 
 we also have:
$$
{\rm deg}\ Z = 2(g-3)-\beta+4.
$$
So we finally get $\beta = g+3-\lambda$. 

The dimensions of the strata are given by the formula in the Main Theorem of \cite{BSa}, which states that
$$
\dim \overline {\mathcal M}^{\lambda} _g = g+2\lambda + 1,
$$
if $\lambda \le \lceil {g\over 2}\rceil$,
 which corresponds to the conditions $\beta \ge (g+5)/2$ if $g$ is odd, or $\beta \ge (g+4)/2$ if $g$ is even. We get
$$
\dim \overline {\mathcal M}^{\lceil {g\over 2}\rceil} _g = \left\{
\begin{array}{ll}
2g+2 & {\rm if}\quad g {\rm odd},\\
2g+1 & {\rm if}\quad g \ {\rm even},\\
\end{array}
\right .
$$
while $\dim \mfour = 2g+3$.

%The codimensions of the Maroni strata can be checked by using the formulas for the dimensions
%of the strata $\overline {\mathcal M}^{\lambda} _g (a,b)$, corresponding to the fourgonal curves
%with fixed invariant $\lambda$ and splittig type of the gonal bundle $(a,b, g-3-a-b)$ given in
%\cite[Theorem 10.11, Theorem 10.15,
%Corollary 10.17, Theorem 11.1]{BSa}.

Finally, we observe that the condition $t\ge 1$ on
the invariant $t$ introduced in \cite[Definition of page 13]{BSa} defines a proper subscheme of $\mfour$
by \cite[Theorem 11.1]{BSa}.
\end{proof}

%\begin{rem}
%The codimensions in the simple ramification case are computed also in \cite[Proposition 4.6]{DP2}.
%\end{rem}
%

Next we consider the divisors in $\mfour$ given by the closures of the loci corresponding to
smooth fourgonal curves with one triple ramification point or two simple ramification
points on the same fiber of the cover of $\mP^1$.
%, following the notation introduced by
%\cite[Subsection 2.2]{DP2}.

\begin{defn} \label{luoghidaevitare}
We set
$$
T:= \overline  { \{[C]\in \mfour | 3p +q\in g^1_4\} },
$$
the divisor of triple ramification, 
$$
D:=\overline  { \{[C]\in \mfour | 2p + 2q\in g^1_4\} },
$$
the divisor of double simple ramifications.

Moreover, we introduce the proper subscheme 
$$
\Upsilon
\subset \mfour
$$ 
corresponding to 
%the closure of the locus
%of curves with either a total ramification point over an assigned point of $\mP^1$, or a triple ramification over an assigned point of $\mP^1$,
%or to 
points of the boundary divisor $\delta\subset \mfour$ given by stable curves with a ramification in a singular point, or by stable curves with two or more nodes. 

%Observe that the first two conditions can be described locally by looking to a birational model of a fourgonal curve on a suitable ruled surface, and locally they correspond to two linear conditions.

%$$
%Q:=\overline  { \{[C]\in \mfour | 4p\in g^1_4\} },
%$$
%the codimension two subscheme of total ramifications, and
%by 
%$$
%A
%$$ 
%the codimension two subscheme of triple ramifications over an assigned point of $\mP^1$.
\end{defn}

In the next theorem we finally present our bounds of $c_2 (\F)$ in terms of $c_1 (\E)^2$. 
We denote by $F_{b}$ and by $L_{b}$ the fiber of $f\colon S \to B$ and respectively 
of $\pi_B \colon Y \to B$ over the point $b\in B$. We also denote by 
 $\pi_b \colon F_b \to L_b \cong \mP ^1$ the restriction to $F_b$ of the degree four cover $\pi\colon S\to Y$. 
%Clearly if $F_b$ has a {\it{unique}} $g^1_4$ then it gives $\pi_b \colon F_b \to L_b$.
%

We recall that the condition $(\dagger)$ has been defined in section \ref{ourresult}.

\begin{thm} \label{boundsc2f} Let $g\geq 10$. Let $B\subset\mfour$, and assume that
$B\not\subseteq T$, $B\not\subseteq D$.
%, $B\cap \Upsilon=\emptyset$.
 Let $f\colon S\to B$ be a semistable fibration 
 which
factorises through a degree $4$ Gorenstein cover.
%$\pi\colon S\to Y$ where 
%%$\pi_B\colon Y\to B$ is 
%of a ruled surface. 
%
%Furthermore, if $Y$ has an irreducible negative section $T_0$ and $T_0$ is contained in the branch divisor, we also assume that there are 
%no triple nor total ramification points over a point of $T_0$.
% 

%Assume that the general fiber $F_b$ is any fourgonal curve.
Then it holds:
\begin{enumerate}
\item if condition $(\dagger)$ is satisfied, then $c_2 (\F) \ge {2\over (g+3)} c_1 (\E)^2$;
\item
if condition $(\dagger)$ is satisfied and if $\pi_b\colon F_b \to L_b $ factorises 
 through a double cover of a hyperelliptic curve of genus $\gamma <{(g-3)\over 6} $, then $c_2 (\F) \ge {\gamma+1 \over (g+3)} c_1 (\E)^2$;
 
\item if condition $(\dagger)$ is satisfied and $\pi_b\colon F_b \to L_b $ does not factorise, then $c_2 (\F) \ge {1\over 6} c_1 (\E) ^2$;
\end{enumerate}

Moreover,
if $g$ is even and $B\not\subset {\rm CE} (\mfour)$ and if condition $(\dagger)$ is satisfied, then
\begin{equation}\label{even} 
\displaystyle c_2 (\F) \ge  {(g+2)\over 4(g+3)} c_1 (\E) ^{2}.
\end{equation} 
If $g$
is odd and $B\not\subset {\rm CE} (\mfour)$, then
\begin{equation}
\label{odd}
c_2 (\F) \ge {1\over 4} c_1 (\E) ^2.
\end{equation}
\end{thm}
\begin{proof} By the assumptions $B\not\subseteq T$, $B\not\subseteq D$, 
the branch locus of $\pi$ is reduced, so $\sF$ is weakly positive outside a zero dimensional scheme 
and outside the branch locus by
Proposition \ref{debolmentepospos}.

If $Y$ admits no negative section, then
the statement of Theorem \ref{boundsc2} $(2)$ holds. Indeed, for any rational positive number $\epsilon$, the linear system $T_Y +(\epsilon - T_Y^2) L$ is $\mQ$-ample,
so a general $\mQ$-divisor in such a system does not contain any component of the branch divisor, hence the proof of Theorem \ref{boundsc2} $(2)$ can be applied. The statements $(1), (2)$ and $(3)$ 
follow by applying the results on the general splitting type of $\sF$ given in Corollary \ref{gensplit}.

 If $Y$ admits a negative section $T_0$, and if $\sF$ is nef over $T_0$,
 then any quotient of $\sF$ is nef over $T_0$, and the proof of Theorem \ref{boundsc2} together with the results of Corollary \ref{gensplit} yield the desired inequalities.

Finally, in the case when $Y$ admits a negative section $T_0$, and if $\sF$ is not nef over $T_0$, we claim that our additional assumption $(\dagger)$ guarantees that 
$\sF(-T_0)$ is nef over $T_0$.

First notice that if $\sF$ is not nef over $T_0$, then $T_0$ is contained in the branch locus $B(\pi)$ of
$\pi$. By the assumption $B\not\subseteq T$, $B\not\subseteq D$, we have that the branch locus is reduced, 
hence we have a first inequality:
\begin{equation}\label{primabasicinequality}
0< (B(\pi) - T_0)\cdot T_0 = (2c_1(\sF) - T_0)\cdot T_0.
\end{equation}
Moreover, as $B(\pi)$ and the branch divisor of the discriminant morphism $\Delta(\pi)$ coincide outside a zero dimensional subscheme, $T_0$ is also in the branch of
$\Delta (\pi)$, so denoting by $p_\sF : \mP(\sF) \to Y$ the natural projection, the divisor $\Delta (S) \subset \mP (\sF)$ satisfies:
\begin{equation}\label{secondabasicinequality}
\Delta (S) \cdot p_\sF^\star  T_0 \sim 2C + C',
\end{equation}
where $C$ and $C'$ are distinct irreducible
sections of $p_\sF^\star  T_0 = \mP (\sF\otimes \sO_{T_0})$.
We also point out that $\Delta (S) \cdot p_\sF^\star  T_0$ contains no fibers of the ruled surface $p_\sF^\star  T_0$
as such fibers would correspond to $\pi$-planar points. Since the $\pi$-fibers over such points are isomorphic to
$Spec \left( {k(y) [u,v] \over (u^2, v^2)} \right)$, we see that any curve containing such a subscheme as
a fiber has at least a node at such a point, with a double ramification. It follows that its stable model belongs either to the boundary of the divisor $D$
of double simple ramifications, or it has two or more singularities, hence this would contradict the condition $(\dagger)$.

Now we observe that 
\begin{equation}\label{terzabasicinequality}
C \cap C' =\emptyset.
\end{equation}
Indeed a point in $C \cap C'$ would imply a total
ramification point for the discriminant cover $\Delta(\pi)$, hence the cover $\pi$ would have either a total ramification point, or a triple ramification point, or a planar point. Indeed, the only pencils of
irreducible conics admitting only one reducible conic are the
hyperosculating and the osculating ones. Furthermore, the only pencil of reducible conics yielding a zero dimensional Gorenstein base locus is the one corresponding to a planar point. The condition $(\dagger)$ implies that 
all this cases can not occur. 

Now we show that by the relations (\ref{primabasicinequality}), (\ref{secondabasicinequality}), (\ref{terzabasicinequality}) it follows that
$\sF(-T_0)$ is nef over $T_0$. 
Since $\sF \otimes \sO_{T_0}$ is not nef, denoting by $H$ the tautological divisor of $\mP(\sF \otimes \sO_{T_0})$,
and by $H_0$ a section of minimal selfintersection, we have
$$
H \cdot H_0 <0,
$$
so $H$ contains $H_0$.
But as the restriction of $H$ to the discriminant $\Delta$ 
gives the ramification divisor of the discriminant morphism, 
and since the only section contained in the ramification
divisor over $T_0$ is $C$, we have
$C=H_0$.

Our claim is equivalent to show that $H +(-T_{0}^{2})L$ is nef where $L$ is a ruling of $\mP(\sF \otimes \sO_{T_0})$. In other words, we want to show that 
$(H -p_\sF^\star  T_0)\cdot A\geq 0$ for any irreducible reduced curve $A\subset \mP(\sF \otimes \sO_{T_0})$. 

Since the Neron-Severi group of $\mP(\sF \otimes \sO_{T_0})$ is $\mathbb Z[H_0]\oplus\mathbb Z[L]$, we have only to prove that
$(H-p_\sF^\star  T_0)\cdot H_0 \ge 0$.

%consider the case of an irreducible curve $A$ 
% such that the restriction of $\pi_0\colon \mP(\sF \otimes \sO_{T_0}) \to T_0$ to $A$ is surjective. Clearly we have only to show that 
%$(H +(-C_{\infty}^{2})L)\cdot A\geq 0$ for any irreducible reduced section of  $\pi_0\colon \mP(\sF \otimes \sO_{T_0}) \to T_0$. 
%Let $A$ be a section of $\pi_0\colon \mP(\sF \otimes \sO_{T_0}) \to T_0$. Then $A^2=2H\cdot A -c_1(\sF)\cdot T_0$. Consider
%$$
%2(H +(-T_{0}^{2})L)\cdot A= 2H\cdot A +2(-T_{0}^{2})=A^2+c_1(\sF)\cdot T_0+2(-T_{0}^{2})
%$$
%\noindent

Note that by construction, the discriminant surface $\Delta(S)
\subset \mP(\sF)$ is a reduced effective divisor, with divisor class 
%(\cite[Definition 4.1]{C})
$$
\Delta (S)\sim 3T_{\mP(\sF)}-\pi^\star c_1(\sF),
$$
where $T_{\mP(\sF)}$ is the tautological divisor of $\mP(\sF)$,
 and 
 $\mP(\sF \otimes \sO_{T_0})$ is not a subdivisor of $\Delta (S)$. In particular 
$$
 \Delta (S)_{| \mP(\sF \otimes \sO_{T_0})}=2H_0+C' \equiv 3H+ (-c_1(\sF)\cdot T_{0})L,
 $$ 
so 
$$
3H \equiv 2H_0 + C' + (c_1(\sF)\cdot T_{0})L.
$$

Assume that $H_0^2\geq 0$. Then by the relation (\ref{primabasicinequality}), we have
$$
3(H -p_\sF^\star  T_0)\cdot H_0= 2 H_0^2 +C' \cdot H_0 +c_1(\sF)\cdot T_0- 3  T_0^2
\geq 2H_0^2 -\frac{5}{2}T_{0}^{2}>0.
$$
\noindent
Finally we consider the case where $H_0^2<0$ and $H\cdot H_0< 0$. We have
$$
[H]=[H_0]+b[L]
$$
where $b\in\mathbb Z$. Observe that by the relation (\ref{secondabasicinequality}) it follows that:
$$C'\sim 3H+ (-c_1(\sF)\cdot T_{0})L-2C=C+(3b+(-c_1(\sF)\cdot T_{0}))L.$$

By our condition (\ref{terzabasicinequality}) we have that
$H_0^2=c_1(\sF)\cdot T_{0}$.
By the standard equality $ c_1(\sF)\cdot T_{0}=H^2=H_0^2+2b$ it follows that $b=0$.
This means that
$H=H_0$, so 
$$
2(H +(-T_{0}^{2})L)\cdot H_0= 2(H_0^2-T_{0}^{2})=(2c_1(\sF)-T_0)\cdot T_0 -T_{0}^{2}\geq -T_{0}^{2}>0.
$$

Now we use the fact $\sF(-T_0)$ is nef over $T_0$ to get our bounds.

 More precisely, by choosing $T_{\wdt Y} = T_0$ in the sequence (\ref{brosius_shift}),
by shifting it by $\sO_Y (-T_0)$ and by restricting it to $T_0$, since
any quotient of $\sF(-T_0) \otimes \sO_{T_0}$ is nef, instead of the inequality \ref{m} we get
the inequality
$$
\deg ( \pi_B^\star M)  \ge -(\alpha -1)T_{0} ^2,
$$
which allows to conclude in the same way as we have $\alpha \ge 4$ by Corollary \ref{gensplit}.

Finally, assume that $B\not\subset{\rm CE} (\mfour)$. Then $\sF$ is balanced. Again by Theorem \ref{boundsc2} the claimed bounds (\ref{even}) and (\ref{odd}) follow. Observe that for the proof of 
the odd balanced case we need not the assumption $(\dagger)$.
\end{proof}

%\begin{cor}\label{R2}Let $f\colon S\to B$ be a semistable fibration induced by $B$ and which
%factorises through a degree $4$ Gorenstein cover $\pi\colon S\to Y$ where 
%$\pi_B\colon Y\to B$ is a ruled surface. 
%Let $R$ be the $\pi$-relative canonical divisor. If $Y$ has a negative section $T_Y$, assume that $T_Y \not \subset B(\pi)$ where $B(\pi)= \pi_\star R$.
%
%Then
%$$
%R^2 > {4\over (g+3)} c_1 (\sE)^2.
%$$
%\end{cor}
% \begin{proof}
% By formula (\ref{formulaR2}) we have $R^2= 2c_1 (\E)^2 -4c_2 (\E) +c_2 (\F)$. 
% By Corollary \ref{upperc2} we have $c_2(\sE) < \frac{g+2}{2(g+3)}c_1 (\sE)^2$
% and by Theorem \ref{boundsc2} (1) we also have $c_2 (\F) \ge {2\over (g+3)} c_1 (\E)^2$. These bounds give
% $
% R^2 > {4\over (g+3)} c_1 (\sE)^2.
% $
%  \end{proof}

The next result will allow to establish a bound on $c_2(\sE)$.
\begin{lem}\label{index}
 Let $\pi\colon S \to Y$ be a Gorenstein cover of arbitrary degree $n$, where $S$ is a normal surface and let $R$ be the $\pi$-relative canonical divisor.
 Let $\sE$ be the reduced direct image sheaf.
Then
$$
R^2 \le {4\over n} c_1 (\sE)^2.
$$
\end{lem}
\begin{proof} 
Let $H$ be any ample divisor on $Y$. Since $\pi$ is finite, the divisor $\pi^\star H$ is ample on $S$. Observe that the $\mQ$-divisor 
$R-{2\over n} \ \pi^\star c_1 (\sE)$ satisfies
$$
\left( R-{2\over n} \ \pi^\star c_1 (\sE)\right) \cdot \pi^\star H = (\pi_\star R -2 c_1(\sE))\cdot H =0.
$$
By the Hodge Index Theorem we have
$$
\left( R-{2\over n} \ \pi^\star c_1 (\sE)\right)^2 \le 0,
$$
which gives the claim.

%By the Algebraic Index Theorem on a normal projective surface the signature of the intersection pairing on 
%${\rm{NS}} (S)_\mathbb R:={\rm{NS}} (S)_\mathbb Q \otimes_\mathbb Q \mathbb R$ is $(1,m-1)$, where $m$ is the dimension of ${\rm{NS}} (S)_\mathbb Q$.
% Consider the subvector space $V$ of ${\rm{NS}} (S)_\mathbb R$
%generated by $R$ and $\pi ^\star c_1 (\sE)$. If ${\rm{dim}}_{\mathbb R}V=1$, then $R\equiv q \pi^\star c_1 (\sE)$ for some $q\in \mR$. By lemma \ref{scriviamolo} (\ref{R}) $\pi_\star R \equiv 2 c_1 (\sE)$, then we have $q =2/n$, so $R^2 = {4\over n} c_1 (\sE)^2$. 
%
%Now assume $V=\mathbb R\langle R \rangle \oplus 
% \mathbb R \langle \pi ^\star c_1 (\sE)\rangle$. 
%We have that $R^2 > 0$ by Proposition \ref{R2}. So the determinant of the matrix 
%$$
%\left(
%\begin{array}{cc}
%R^2 & R\cdot \pi^\star c_1 (\sE)\\
%R\cdot \pi^\star c_1 (\sE) & (\pi^\star c_1 (\sE))^2\\
%\end{array}
%\right)
%$$
%is non positive, that is $R^2 (n c_1(\sE)^2)\le (R\cdot \pi^\star c_1 (\sE))^2$.
%As $\pi_\star R \equiv 2 c_1 (\sE)$, the claim follows.
\end{proof}

%\begin{cor} In the same hypotheses of Corollary \ref{R2}, we always have
%$$
%c_1 (\sE)^2 >0.
%$$
%\end{cor}
%
%
%\begin{proof} By Corollary \ref{R2} and by Lemma \ref{index} we have
%$$
%{4\over g+3} c_1 (\sE)^2 < R^2 \le c_1(\sE)^2,
%$$
%and such inequalities imply $c_1(\sE) ^2 >0$.
%
%\end{proof}
%

\begin{lem}\label{indexapplied} Let $f\colon S\to B$ be a genus $g$ fibration which 
factorises through a degree $4$ Gorenstein cover $\pi\colon S\to Y$ where $Y$ is a smooth surface. Then
\begin{equation}
c_2(\E) \ge {1\over 4} \Big(c_1 (\E) ^2 + c_2 (\F)\Big).
\end{equation}
\end{lem}

\begin{proof} By Proposition \ref{invarianti} we have $R^2= 2c_1 (\E)^2 -4c_2 (\E) +c_2 (\F)$. 
Since by Lemma \ref{index} it holds $c_1 (\sE)^2\ge R^2$, the claim follows.
\end{proof}

Now we can apply Theorem \ref{boundsc2f} to a factorised fibration.
\begin{cor} \label{formulaslosloslo} Let $f\colon S\to B$ be a genus $g$ fibration which 
factorises through a degree $4$ Gorenstein cover $\pi\colon S\to Y$ where 
$\pi_B\colon Y\to B$ is a ruled surface. Then
\begin{equation}\label{formulaslope}
\displaystyle s(f)\ge { {(g-1)\over (g+3)} c_1 (\sE)^2 \over {(g+1)\over 4(g+3)} c_1 (\sE)^2 - {1\over 4} c_2 (\sF)}.\end{equation}

\end{cor}
\begin{proof}
Consider the following function obtained in Corollary \ref{invariantiinvarianti}:
\begin{equation}\label{eq:pendenzaclassi}
s( c_1 (\sE)^2, c_2(\sF),c_2(\sE) ):= { {2(g+1)\over (g+3)} c_1(\sE)^2 - 4c_2 (\sE)+c_2 (\sF)  \over {(g+2)\over 2(g+3)} c_1 (\sE)^2 - c_2 (\sE)}.
\end{equation}

Observe that the partial derivative of $s( c_1 (\sE)^2, c_2(\sF),c_2(\sE) )$ with respect to 
$c_2(\sE)$ is positive if and only if $c_2(\sF) \ge {2\over (g+3)} c_1(\sE)^2$. This is always satisfied by Theorem 
\ref{boundsc2f} (1). Hence we can use the bound on $c_2(\sE)$ given in Lemma \ref{indexapplied}, and we conclude by applying Lemma \ref{indexapplied}.
\end{proof}

%%%%%%%%%%%%%%%%%%%%%%%%%%%%%%%%%%%%%%%%%%%%%%%%%%%%%

\subsection{Proof of the Main Theorem for factorised fibrations}
\label{momomo}

Now it is easy to see that if $S$ is a {\it{normal surface}} and $f\colon S \to B$ is a semistable fibration 
which factorises through a finite Gorenstein cover, the claims of the main theorem are 
straight consequence of Theorem \ref{brusac}, of Theorem \ref{boundsc2f}, of Corollary \ref{formulaslosloslo}, and the fact that such a bound is an increasing function in $c_2 (\F)$.\qed

%%%%%%%%IL TEOREMA PRINCIPALE%%%%%%%%%%%%%%%%%%%%

\section{The main theorem}

\subsection{A theorem on $n$- gonal semistable fibrations}

Let $B\subset \mngon$ be a curve, which is not contained in the boundary divisor, and let $f:S \to B$ be the semistable fibration
associated with $B$.

%of the $1$-dimensional family we can normalise $B$ and we can build a semistable fibration $f\colon S_1\to B_1$
% such that it is defined a morphism $m(f)\colon \overline B_1\to \overline  {\it M}^{1}_{g,n}$ dominant over $B$. Moreover letting $b\in B_1$ a general point 
% $m(f)\colon b\mapsto [F_b]$ where $F_b=f_1^{-1}(b)$. Let us recall some facts about a semistable reduction of a
%curve $B\subset \mngon$,
%

We recall the following facts, which are well-known to the 
experts: 

\begin{thm}\label{fattorizzazione} Let $f\colon S\rightarrow B$ be a non isotrivial semistable 
fibration such that its general fiber is smooth non-hyperelliptic and of genus $g>2$.
% Let $B\subset \overline {\it  M}^{}_{g}$ be the image of $B_1$ by the moduli morphism 
%$m(f_1)\colon B_1\to \overline{\it  M}^{}_{g}$. Assume that the general fiber $F$ is a smooth $n$-gonal curve. Then there exists a base 
%    change ${\overline{B}}\rightarrow B_{1}$ such that if
%     $f\colon S\rightarrow \overline B$ is the pull-back fibration then $s(f)=s(f_{1})$. 
     
%     Moreover
 Then, by possibly replacing $B$ by a finite base-change, there exists a ruled surface $\pi_{B}\colon\mathbb 
  \mP(\B)\to B$ and a rational degree $n$ map $\rho\colon S\dashrightarrow \mP(\B)$, such that $f\colon S\to B$ is factorised through $\rho\colon S\dashrightarrow \mP(\B)$. 
 Finally, $S$ is a minimal surface of general type.
 \end{thm}

\begin{proof}
The proof follows by the theorem of semistable reduction and by standard facts of the theory of surfaces; see \cite[Section 2, The basic construction, page 337]{HMo}.
\end{proof}
By Theorem \ref{fattorizzazione}, in order to evaluate the slope 
of a non isotrivial semistable fibration with fiber an $n$-gonal curve it can be useful to study the following diagram:

\begin{equation}
\label{eq:familyB1}
\xymatrix{X \ar[d]_{\tau}
\ar[r]^{\pi}  &  \mP(\B) \ar[d]^{\pi_{B}}\\
S \ar[r]^{f} &  B,} 
\end{equation}\noindent
where $S$ is a minimal surface of general type, in particular $K_{S}$ 
is big and nef, and $\tau\colon X\to S$ is the minimal resolution of 
$\rho\colon S\dashrightarrow \mP(\B)$.
We denote by $Y$ the ruled surface $\mP(\B)$.

We observe now that the morphism 
${\widehat{\pi}}\colon {\widehat{X}}\to Y$ arising from the Stein factorisation of $\pi\colon X\to Y$ is not necessarily a Gorenstein cover. The non Gorenstein zero dimensional schemes
of degree four are isomorphic to one of the following (see \cite [Tables 6.1 and 6.2]{HaMi}:
$$
{\rm Spec} \ K \oplus {K[x,y]\over (x,y)^2}, \quad {\rm Spec} \ {K[x,y]\over (x^3,xy,y^2)}, \quad {\rm Spec} \ {K[x,y,z]\over (x,y,z)^2}.
$$
Hence the possible non Gorenstein fibers are supported on total or triple ramification points for the cover $\widehat \pi$, and since they are not curvilinear, they
are singular points of the corresponding fiber in the fibration $\widehat X \to B$. The stable model of such a fiber is either a curve with two or more nodes, or a curve with a node in $P$, and such that one or both branches of the curve are ramified 
in $P$. So such fibers occur on curves corresponding to non general points of the boundary divisor $\delta$. It follows that if the starting curve $B 
\subset \mfour$ does not interesect a certain proper subscheme $\Xi \subset \mfour$, the Stein factorisation of $\pi\colon X\to Y$ is a Gorenstein 
cover.

Nevertheless, it is possible to treat a more general case, by imposing some condition on the possible non Gorenstein fibers of $\widehat \pi:\widehat X \to Y$.
 
%On the other hand we are interested to bound the slope $s(B)$ of a curve in a moduli space and we want hypotheses
%that can be satisfied at least by curves passing through the general point $[F]\in\overline  {\it M}^{1}_{g,n}$.

 From now on and with the notation of diagram \ref{eq:familyB1} we assume that the following condition holds:
   \medskip

   $(\star)$ the possible non Gorenstein points of $\widehat X$ occur only over simple nodes of $B(\pi)$.
    \medskip
    
     We point out that the sweeping families constructed by Harris and Morrison in \cite{HMo} satisfy our condition (cf. \cite{BeZ}).

%   the morphism 
%   $\pi\colon X\rightarrow Y$ is a {\it{finite morphism}} 
%   over a neighbourhood of any point of $B(\pi)$, which is {\it{not}} 
%   a simple node.

  \begin{defn}\label{defonozioneHM}
 A {\it curve with good Gorenstein factorization} is a smooth curve $B\subset \mngon$ such that the condition $(\star)$ holds.
 \end{defn}

Next we shall determine the invariants of one dimensional families $B$ of fourgonal curves with good Gorenstein factorization. 
 
\subsection{The Gorenstein model}

We consider again the diagram (\ref{eq:familyB1}). Let $X\to{\widehat{X}}\to Y$ be the Stein factorisation of $\pi\colon X\to Y$. 
The technical side of Definition \ref{defonozioneHM} is that the singular points of ${\widehat{X}}$ which are not Gorenstein singularities occur only over the points of the branch locus
$B(\pi)$ of $\pi\colon X\to Y$ where $B(\pi)$ has a node. We want to replace the finite Stein morphism ${\widehat{\pi}}\colon {\widehat{X}}\to Y$ with a Gorenstein morphism 
${\widetilde{\pi}}\colon{\widetilde{X}}\to {\widetilde{Y}}$ where ${\widetilde{X}}$ is a Gorenstein model of ${\widehat{X}}$ and ${\widetilde{Y}}$ is obtained by suitable blow-ups of the ruled surface $Y$.

Let us study closely what happens over a node $p$ of $B(\pi)$ if the analytical germ of ${\widehat{X}}$ over a germ of $Y$ centered 
at $p$ has a non Gorenstein singularity. Essentially we need to understand the 
local behaviour of the  cover ${\widehat{\pi}}\colon {\widehat X} \to Y$ over an analytical 
neighboorhood of $p$ isomorphic to the polydisk $\Delta\subset\mathbb 
C^{2}$. We will follow \cite{L}. 
We know that if ${\widehat{\pi}}\colon {\widehat X} \to Y$ is a finite and dominant morphism
from a variety ${\widehat X}$ with isolated singularities to a smooth variety $Y$ and 
there is a simple normal crossing divisor $D$ in $Y$ such that ${\widehat{\pi}}\colon {\widehat X} \to Y$ is smooth
over $Y \setminus D$, then ${\widehat X}$ has algebraic abelian quotient 
singularities. In our case to analyse the behaviour at $p\in B(\pi)\subset Y$ we
can locally identify $B(\pi) \subset Y$ with $\{xy=0\} \subset {\mathbb{C}}^2_{x,y}$ and $p$ with $(0,0)$.
The local fundamental group of $\mC ^2_{x,y} \setminus \{xy=0\}$ near the origin is
free abelian of rank $2$, generated by loops around the two boundary
divisors. Thus, away from the branch locus, ${\widehat{\pi}}\colon {\widehat X} \to Y$ is locally analytically equivalent to a covering of the 
form ${\mathbb{C}}^2_{u,v}\setminus\{u
v=0\}\rightarrow{\mathbb{C}}^2_{x,y}\setminus\{x y=0\}$ 
determined by the corresponding morphism of fundamental groups. Since 
the degree is $4$, we have to consider subgroups of index $\leq 4$. Explicitly,
let $N={\mathbb{Z}}^2$ be the lattice of 1-parameter subgroups of the torus
$({\mathbb{C}}^{*})^2_{x,y}$.

Let $N' \subset N$
be a sublattice of index $4$. Then the
inclusion $N' \subset N$ corresponds to a degree $4$ toric morphism 
$Z\rightarrow{\mathbb {C}}^2$ ramified along the toric boundary 
divisors. Since the degree is $4$ the only case
where $Z$ is not Gorenstein is given by $Z= 1/4(1,1) = 
{\mathbb{C}}^2_{u,v}/({\mathbb{Z}}/4{\mathbb{Z}})$
with ${\mathbb{Z}}/4{\mathbb{Z}}$-action $(u,v) \mapsto (iu,iv)$, and $Z\rightarrow \mathbb 
C^{2}$ given by $(u,v) \mapsto (x,y)=(u^4,v^4)$. Note that in this case, 
above $p$ there exists a unique point $q\in{\widehat X}$ of total ramification.
Now if $\widetilde{Z} \rightarrow Z$ is the minimal resolution and 
$\widetilde{Y} \rightarrow {\mathbb {C}}^{2}$
is the blowup at $p$, then $\widetilde{Z} \rightarrow 
\widetilde{Y}$ is finite flat. Now we can glue the local partial resolution of the quotient 
singularities $q_1$,..., $q_s$ of type $1/4(1,1)$ inside ${\widehat {X}}$ to construct a partial resolution 
$\tau\colon \widetilde X \to {\widehat X}$ such $\widetilde X$ is Gorenstein over those singularities. Let ${\wdt{Y}}\to Y$ be the germ of the blow-up of $Y$ at the images of the point
$q_1$,..., $q_s$. There exists a germ of a morphism ${\wdt{\pi}}\colon{\wdt{X}}\to{\wdt{Y}}$ such that if $E_{q_i}$, $i=1,..., s$
 denote the corresponding exceptional divisors on $\widetilde X$ and if $E'_i := \wdt \pi (E_{q_i})$ for $i=1,\dots,s$, then it holds that 
\begin{enumerate}
\item $2K_{\widetilde 
X}=\tau^{\star}(2K_{X})-\sum_{i=1}^{s}E_{q_{i}}$ 
where each 
$E_{q_{i}}$ is a rational curve;
\item  $E_{q_{i}}^{2}=-4$ and 
$\widetilde \pi : E_{q_{i}}\to E $ is $4:1$ on $E_{q_{i}}$, $i=1,\ldots , s$;
\item
$E_{q_i} \cdot \widetilde R =6$ and $E_{q_i} \not\subset \widetilde R$,
where $\widetilde R$ denotes the ramification divisor of $\widetilde \pi$;
\item $K_{\widetilde X}= \tau ^\star K_{X} -{1\over 2} E_{q_i}$ as $\mathbb Q$- divisors and locally over $q_i$.
\end{enumerate}
For the last equation above note that as $\mathbb Q$- divisors we can write locally
$K_{\widetilde X}= \tau ^\star K_{X} +a E_{q}$ and
$
-2=E_{q}^{2} + a E_{q} ^{2} = -4-4a,
$ hence
$
a=-{1\over 2}.
$
\medskip

Now we perform the analysis of a sublattice $N' \subset N$ of index $3$.
This means that over the node $p\in B(\pi)$ we have 
an open set of ${\widehat{X}}$ which is mapped $1$-to-$1$ and another which is mapped $3$-to-$1$. The only non-Gorenstein possibility is 
given by $Z= 1/3(1,1) = 
{\mathbb{C}}^2_{u,v}/({\mathbb{Z}}/3{\mathbb{Z}})$. Denote by $E_{a}\subset 
{\widetilde{X}}$ the exceptional divisor over a non-Gorenstein point $a\in {\widehat{X}}$ of this kind. It follows that 
$E_{a}$ is an irreducible reduced rational curve with selfintersection
$E_{a}^{2}=-3$, and there is a $3$-to-$1$ cover $E_{a}\rightarrow E''$ of the exceptional curve 
$E''\subset{\widetilde{Y}}$ given by blowing-up at $p$. Moreover we have
$E_{a}\cdot\widetilde R=4$. Notice that to resolve the morphism we 
have to add a $-1$-curve $A_{a}\subset\widetilde X$ which is disjoint from $E_{a}$ and 
which is mapped $1$-to-$1$ to $E''\subset \widetilde Y$. Notice that 
$\widetilde{R}\cdot A_{a}=0$ but if we contract $A_{a}$ we can't 
factorise the cover through $\widetilde Y$. We shall denote by $E'' _1, \dots, E'' _t$ all the exceptional curves on $\wdt Y$ of this type obtained by the images of the point 
$a_1,..., a_t\in{\wdt{X}}$ of index $3$.

Finally, if $N' \subset N$
is a sublattice of index $2$, then we have that $\widehat X$ is 
Gorenstein above $p$.

The given description allows to pass from ${\widehat{\pi}} \colon {\widehat{X}}\to Y$ to a Gorenstein cover.

%%%%%%%%%%%%%%GORENSTEINFOUR%%%%%%%%
\begin{thm}\label{gorensteinfour} We use the above notation.
Let $\pi\colon X\rightarrow Y$ be a generically finite degree four cover between a normal surface $X$ and a ruled surface 
$\pi_B\colon Y\to B$ over a smooth curve $B$ such that $f_{X}\colon X\to B$ is a genus $g\geq 10$ fibration where 
$f_{X}:=\pi_B\circ\pi$. 
Let ${\widehat{\pi}}\colon {\widehat{X}}\to Y$ be the finite morphism of degree four given by the Stein factorisation $X\to{\widehat{X}}\to Y$ of 
$\pi\colon X\rightarrow Y$. Assume that the condition $(\star)$ holds for $\widehat{\pi}$.
 
 Then there exists a finite number of blow-ups 
$\sigma \colon \widetilde  Y \to Y$ and a partial resolution 
$\wdt \tau\colon \widetilde X\to {\widehat{X}}$ of $\widehat X$, such that $\widetilde X$ 
is Gorenstein and there exists an induced degree four Gorenstein cover 
$\widetilde \pi : \widetilde X \to \widetilde Y$ such that $\sigma\circ{\wdt{\pi}}={\widehat{\pi}}\circ{\wdt\tau}$.

Moreover, the possible non 
Gorenstein points of ${\widehat{X}}$ are total ramification points or 
ramification points of index three for ${\widehat{\pi}}\colon {\widehat{X}} \to Y$. Let 
$q_1$,..., $q_{s}\in{\widehat{X}} $ be the non Gorenstein points of total ramification, and let $a_1,...,a_{t}\in {\widehat{X}}$ be the non Gorenstein index $3$ ramification points. 
Then the ${\wdt{\pi}}$-relative canonical divisor $\widetilde R$ satisfies 
\begin{equation}\label{ingiusto}
{\widetilde \pi{ _\star}}{ \widetilde R} \equiv 2 \left((g+3) T_0 + {c_1 ^2 (\widetilde \sE)+9s+4t \over 2 (g+3)} \wdt L -3
\sum_{i=1}^s  E'_i -2\sum _{j=1}^t E''_j\right),
\end{equation}
where $T_0 = T_{\widetilde Y} - {1\over 2} T_{\widetilde Y} ^2 {\wdt{L}}$, $T_{\widetilde Y}$
is any section of ${\widetilde Y}$, $\wdt L$ is a general fiber of $\widetilde Y\to   B$,
$E'_i$, $i=1,...,s$ are $(-1)$-curves arising from the images of total ramification non-Gorenstein
points and $E''_j$, $j=1,..., t$ are $(-1)$-curves arising from the images of index $3$ non-Gorenstein points.
\end{thm}

\begin{proof} The analysis performed above on the non Gorenstein points of $\widehat X$ shows that $\sigma\circ{\wdt{\pi}}={\widehat{\pi}}\circ{\wdt\tau}$.

We show the formula for ${\widetilde \pi{ _\star}}{ \widetilde R}$. By Theorem \ref{scriviamolo} (\ref{R}) 
and by the fact that 
${\rm{NS}}_{\mathbb Q}(\wdt Y)=T_{0}\mathbb Q\oplus \wdt L\mathbb Q\oplus E'_1\mathbb Q 
\oplus...\oplus E'_s\mathbb Q \oplus E''_1\mathbb Q ...\oplus E''_t \mathbb Q$ we can write:

$$
c_1 (\widetilde \sE ) \equiv {1\over 2} \widetilde \pi_\star \widetilde R\equiv  (g+3)T_0
+d \wdt L +\sum_{i=1}^s a_i E'_i  +\sum _{j=1}^t b_j E''_j.
$$ 
Set $c_1 ^2:=c_1  ( \widetilde \sE)^2$.
We get
$$
d= {c_1 ^2 +\sum a_i ^2 +\sum b_j ^2\over 2(g+3)}.
$$
We recall that $\widetilde\pi ^\star E'_i= E_{q_i}$, $i=1,..., s$ and $ \widetilde\pi ^\star E''_j = E_{a_j} + A_{a_j}$, $j=1,..., t$. It follows that if $i\leq k\leq s$
$$
6= \widetilde R \cdot E_{q_k} =\widetilde R  \cdot \widetilde\pi ^\star E'_k=
\widetilde\pi_\star \widetilde R \cdot E'_k=
$$
$$
=2 \left( (g+3)T_0
+{c_1 ^2 +\sum_{i=1} ^s a_i^2 +\sum _{j=1}^t b_j^2 \over 2(g+3)} \wdt L +\sum_{i=1}^s a_i E'_i 
+\sum_{j=1}^t b_j E''_j \right) \cdot E'_k = -2a_k,
$$
hence
$
a_k =-3
$ for any $k=1,\dots,s$.
Similarly if $1\leq l\leq t$ we find
$$
4= \widetilde R \cdot (E_{a_l}+E'_{a_l}) =\widetilde R  \cdot \widetilde\pi ^\star E''_l=
\widetilde\pi_\star \widetilde R \cdot E''_l=
$$
$$
=2 \left( (g+3)T_0
+{c_1 ^2 +\sum_{i=1} ^s a_i^2 +\sum_{j=1}^t b_j^2 \over 2(g+3)} \wdt L+\sum_{i=1} ^s a_i E'_i 
+\sum_{j=1}^t b_j E''_j \right) \cdot E''_l = -2b_l,
$$
hence
$
b_l =-2
$ for any $l=1,\dots,t$.
\end{proof}

We can apply Theorem \ref{gorensteinfour} to the case given in diagram (\ref{eq:familyB1}). Actually we can prove a slightly more general theorem.

\begin{thm}\label{theoremsnfour} With the same assumptions as in Theorem
\ref{gorensteinfour}, the Gorenstein resolution $\wdt\tau\colon\wdt X\to{\widehat{X}}$ induces a fibration 
$\wdt f  \colon \wdt X \to  B$, which slope $s(\wdt f)$ satisfies the claims of the Main Theorem stated in the Introduction.
\end{thm}

\begin{proof} Since $B(\pi)$ satisfies the $(\star)$-condition then we can apply Theorem \ref{gorensteinfour} to ${\widehat{\pi}}\colon {\widehat{X}} \to Y$. Hence 
let $\wdt \tau\colon \widetilde X\to {\widehat{X}}$ be the partial resolution of $\widehat X$, such that $\widetilde X$ 
is Gorenstein. Again by Theorem \ref{gorensteinfour} there exists an induced degree four morphism $\widetilde \pi : \widetilde X \to \widetilde Y$ 
where $\widetilde Y\to Y$ is obtained by blowing up the images of the non Gorenstein index $3$ points or of non Gorenstein total ramification points of $\widehat X$.

By Proposition \ref{Casnati}, if $\wdt R=K_{\wdt X/\wdt Y}$ there exists an embedding $\wdt j\colon \wdt X\to  \mP (\wdt \sE)$ where 
$\sO _{\wdt X} (R) \cong j^\star \sO _{\mP (\wdt \sE)} (1)$. We denote by $\wdt \sF= ({{\wdt \pi} _{\mP}})_\star \sI _{\wdt X, \mP (\wdt \sE) }(2)$ the bundle of conics.
We recall that by the proof of Proposition \ref{slopegenerale} it also holds:
$$
\chi _{\wdt f}=\chi (\sO_{\wdt X}) - (g-1)(g(  B) -1)= 4\chi(\sO _{\wdt Y}) +
{1\over 2} c_1 (\wdt \sE) \cdot K_{\wdt Y} + {1\over 2} c_1 (\wdt \sE) ^2 - c_2 (\wdt \sE)-
(g-1)(g(  B)-1).
$$ 
By the equation (\ref{ingiusto}) obtained in Theorem \ref{gorensteinfour} we have:
$$
c_1 (\widetilde \sE) \equiv (g+3) T_0 + {c_1 ^2 (\widetilde \sE)+9s+4t \over 2 (g+3)} \wdt L -3
\sum_{i=1}^s  E_i -2\sum_{j=1}^t E''_j, \quad K_{\wdt Y} \equiv -2 T_0 +{\wdt f}^\star K_{  B} + \sum _{i=1}^s E'_i +\sum_{j=1}^t E''_j.
$$
Then
$$
c_1 (\widetilde \sE) \cdot K_{\wdt Y} = 2(b-1)(g+3) -{c_1 (\widetilde \sE) ^2 \over (g+3)} +3{g\over (g+3)} s + 2 {(g+1)\over (g+3)} t,
\quad K_{\widetilde Y} ^2= -8(g(\overline B-1)-s-t.
$$
It follows that:
\begin{equation}\label{eradafare}
K_{\wdt f}^2 = 2{(g+1)\over (g+3)} c_1 (\widetilde \sE)^2-4c_2 (\wdt \sE) + c_2 (\wdt \sF),
\end{equation}
and 
\begin{equation}\label{eradafare2}
\chi _{\wdt f} = {(g+2)\over 2(g+3)} c_1 (\wdt \sE )^2 - c_2 (\wdt \sE) + {3g\over 2 (g+3)} s +{(g+1)\over (g+3)} t
\end{equation}
As $\wdt X$ is normal, we may apply Lemma \ref{index} to get 
\begin{equation}\label{eradafare3}
c_2 (\wdt \sE) \ge {1\over 4} \Big   (  c_1 (\wdt\sE )^2 + c_2 (\wdt \sF)  \Big),
\end{equation}
which follows from the relation $\wdt R^2 = 2 c_1 (\wdt \sE)^2 - 4 c_2 (\wdt \sE) + c_2 (\wdt \sF)$ shown in Proposition \ref{invarianti}.

By equations (\ref{eradafare}), (\ref{eradafare2}) and by the inequality (\ref{eradafare3}) we obtain:
\begin{equation}\label{slopescoppiata}
s(\wdt f) \ge 4 + { c_2 (\wdt \sF) - {2\over (g+3)}  c_1 (\widetilde \sE)^2  \over 
{(g+1)\over 4(g+3)} c_1 (\widetilde \sE)^2-{1\over 4} c_2 (\wdt \sF)+ {3g\over 2 (g+3)} s +{(g+1)\over (g+3)} t}.
\end{equation}
Set 
$$
\upsilon(c_1 (\wdt \sE )^2, c_2 (\wdt \sF), s,t):=4 + { c_2 (\wdt \sF) - {2\over (g+3)}  c_1 (\widetilde \sE)^2  \over 
{(g+1)\over 4(g+3)} c_1 (\widetilde \sE)^2-{1\over 4} c_2 (\wdt \sF)+ {3g\over 2 (g+3)} s +{(g+1)\over (g+3)} t}.
$$
We observe that $\upsilon(c_1 (\wdt \sE )^2, c_2 (\wdt \sF), s,t)$ is an increasing function in $c_2 (\wdt\F)$. 

Now $c_2 (\wdt\F)$ can be bounded
using similar arguments as in the ruled case. The claim follows by next Proposition \ref{boundsc2scoppiatafour}, as the resulting expressions
are all bounded from below by the same expressions but evaluated in $s=t=0$.
\end{proof}

\begin{prop} \label{boundsc2scoppiatafour}
Let $B\subset\mfour$, and assume that
$B\not\subseteq T$, $B\not\subseteq D$.

Provided that the cover $\widetilde \pi : \widetilde X \to \widetilde Y$
satisfies
condition $(\dagger)$ in cases $(1)$, $(2)$ and $(3)$ below, the second Chern class of the bundle of conics $\wdt \sF$ is bounded by:
\begin{enumerate}
\item $c_2 (\wdt\F) \ge {2\over (g+3)} \left( c_1 (\wdt\E)^2 +9 s+4t \right)$;

\item $c_2 (\wdt\F) \ge {1\over 6} \left( c_1 (\wdt\E)^2 +9 s+4t \right)$ 
if $\pi _{b} \colon F_b \to \mP ^1$ does not factorise, where $F_b$ a general fiber of $\wdt f\colon\wdt X \to B$;

\item $\displaystyle c_2 (\wdt\F) \ge {(g+2)\over 4 (g+3)} \left( c_1 (\wdt\E)^2 +9 s+4t \right)$
 if $g$ is even and $B\not\subset{\rm CE} (\mfour)$;
\item
$c_2 (\wdt\F) \ge {1\over 4} \left( c_1 (\wdt\E)^2 +9 s+4t \right)$ 
if $g$ is odd and $B\not\subset{\rm CE} (\mfour)$.
\end{enumerate}
\end{prop}

\begin{proof} We use Theorem \ref{boundsc2}. The proof is similar to the proof of Theorem \ref{boundsc2f}, taking into account that we can still use for the case of the blow-up surface $\wdt Y$ 
the bounds on $\beta$ given in Corollary \ref{gensplit}.
\end{proof}

\subsection{Proof of the main theorem}\label{laprovavava}

We can finally conclude with our main result. Indeed fix a curve $B\subset\mfour$ as in the main theorem and consider the diagram \ref{eq:familyB1}. Let ${\widehat{\pi}}\colon \widehat X \to Y$ be the finite morphism to the ruled surface $\pi_B\colon Y\to   B$ obtained by the Stein factorisation $X\to{\widehat{X}}\to Y$ of the generically finite morphism $\pi\colon X\to Y$.
The surface $\widehat X$ is normal; see, for instance, \cite[Example 2.1.15, page 126]{La}. By assumption, $B$ is a curve with good Gorenstein factorisation, hence the morphism ${\widehat{\pi}}\colon {\widehat{X}}\to Y$ satisfies the $(\star)$ condition. Then we can apply 
Theorem \ref{theoremsnfour} to ${\widehat{\pi}}\colon{\widehat{X}}\to Y$ and we obtain a fibration ${\wdt{f}}\colon{\wdt{X}}\to  B$ such that the claims hold for the slope of ${\wdt{f}}\colon{\wdt{X}}\to  B$. 

Now we have that ${\wdt{X}}$ is a Gorenstein model of $S$ and by Theorem \ref{fattorizzazione}, $S$ is a minimal surface. Observe that ${\wdt{X}}$ has only canonical singularities;
indeed, $\widetilde X$ coincides with $X$ 
outside the points over the nodes of the branch locus. Moreover, by our resolution 
process the surface $\widetilde X$ has only rational Gorenstein singularities, which
are canonical.

Therefore $K^2_S\geq K^2_{\wdt X}$. This implies that $s(f)\geq s(\wdt f)$ and all the claims of the main theorem follow.\qed

\begin{rem}\label{tesipatel}
A. Patel recently proved in his Ph. D. thesis that
$g\equiv 3 \mod 6$, the slope of a sweeping $4$-gonal family not contained entirely in the divisors $T$, $D$, ${\rm CE} (\mfour)$ and the Maroni divisor 
$M(\mfour)$ consisting of curves with non balanced reduced direct image sheaf, is
bounded below as $s(f)\ge {11\over 2} - {15\over 2g}$. In this particular numerical case, the reduced direct image sheaf of the general fiber
has splitting type $(g+3)/3,(g+3)/3,(g+3)/3$, so the reduced direct image sheaf $\sE$ arising in our construction has non negative Bogomolov discriminant, by \cite[Theorem 2.2.1]{Mo}. The inequalities
$$
c_2(\sE) \ge {1\over 3} c_1 (\sE) ^2, \ c_2(\sF) \ge {1\over 4} c_1 (\sE) ^2,
$$ 
together with our formula \ref{slope4} give the desired bound.
\end{rem}


\begin{thebibliography}{Muk04}
\bibitem[AC]{AC} E. Arbarello, M. Cornalba, \emph{Footnotes to a paper of Beniamino Segre},
 Math. Ann. 256 no. 3, (1981), 341-362.
 
 
\bibitem[ACG]{ACG} E. Arbarello, M. Cornalba, P. Griffiths, 
 \emph{Geometry of Algebraic Curves, Volume II}, A Series of Comprehensive 
 Studies in Mathemtics 268, Springer, (2011). 

\bibitem[AW]{AW} M. Artin, G. Winters, \emph{
Degenerate fibres and stable reduction of curves}, 
Topology 10, (1971), 373-383. 

\bibitem[BS]{BS} M. A. Barja, L. Stoppino, \emph{Slopes of
trigonal fibred surfaces and of higher dimensional fibrations},
Ann. Sc. Norm. Super. Pisa Cl. Sci. (5), n. 4, (2009), 647-658.


\bibitem[BZ]{BZ} M. A. Barja, F. Zucconi, \emph{On 
the slope of fibred surfaces}, Nagoya Math. J. 164,  (2001), 103-131.



\bibitem[BeZ]{BeZ} V. Beorchia, F. Zucconi, \emph{A note on Harris Morrison sweeping families}, Collect. Math. 66 no. 2, (2015), 191-202.

\bibitem[B]{B} J.E.
Brosius, \emph{Rank-2 vector bundles on a ruled surface. I,} Math. Ann. 265, n. 2, (1983), 155-168.

\bibitem [BSa]{BSa} M. Brundu, G. Sacchiero, 
\emph {Stratification of the moduli space of fourgonal curves}, Proc. Edinb. Math. Soc. (2) 57 no. 3, (2014), 631-686. 

\bibitem[CE]{CE} G. Casnati, T. Ekedahl, \emph{Covers of algebraic varieties. I. A general structure
theorem, covers of degree 3,4 and Enriques surfaces}, 
J. Algebraic Geom. 5, n. 3, (1996), 439-460.


\bibitem[C]{C} G. Casnati, \emph{ 
Covers of algebraic varieties III. The discriminant of a cover of degree $4$ and the trigonal construction}, 
 Trans. Amer. Math. Soc. 350, n. 4, (1998), 1359-1378.

%\bibitem[CFM]{CFM} D. Chen, G. Farkas, I. Morrison, \emph{
%Effective divisors on moduli spaces of curves}, preprint (2012),
%arXiv:1205.6138v1. To appear in "A Celebration of Algebraic Geometry"
%Clay Mathematics Proceedings, Volume published on the occasion of Joe 
%Harris' 60th birthday.
%

\bibitem[CH]{CH} M. Cornalba, J. Harris, \emph{Divisor
classes associated to families of stable varieties,
with applications to the moduli space of curves,} 
Ann. Sc. Ec. Norm. Sup. 21, n. 4, (1988), 455-475.

\bibitem [CS]{CS} M. Cornalba, L. Stoppino, 
\emph{A sharp bound for the slope of double cover 
fibrations}, Michigan Math. J. 56, n. 3, (2008), 551-561.


\bibitem [DP] {DP}
A. Deopurkar, A. Patel, \emph{Sharp slope bounds for sweeping families of trigonal curves,}
 Math. Res. Lett. 20, no. 5, (2013), 869-884.

%\bibitem [DP2] {DP2}
%A. Deopurkar, A. Patel, \emph{The Picard rank conjecture for the Hurwitz spaces of degree up to five,}
%Algebra Number Theory 9, no. 2 (2015), 459-492.



\bibitem[DZ]{DZ} P. De Poi, F. Zucconi, \emph{Gonality, apolarity and hypercubics},
Bull. London Math. Soc. vol. 45 n. 2, (2011), 623-645.

\bibitem [E] {E} D. Eisenbud \emph{Commutative Algebra
with a View Toward Algebraic Geometry},
Series: Graduate Texts in Mathematics, Vol. 150, Springer (1995).

%\bibitem [EH] {EH} D. Eisenbud and J. Harris, \emph{
%The Kodaira dimension of the moduli space of curves of
%genus $\ge$ 23}, Invent. Math. 90, (1987), 359-387.

%\bibitem [Fa]{Fa} G. Farkas, \emph{The geometry 
%of the moduli space of curves of genus 23,} Math. Ann. 318, (2000), 43-65.

\bibitem [FJ]{FJ} M. Fedorchuk, D. Jensen, \emph{Stability of 
2nd Hilbert points of canonical curves, } 
Int. Math. Res. Not. n. 22, (2013), 5270-5287. 

\bibitem [Fr]{Fr}
R. Friedman, \emph{Algebraic surfaces and
holomorphic vector bundles,} Universitext, Springer-Verlag, 
New York, (1998).

\bibitem [Fu]{Fu}
W. Fulton, \emph{Intersection Theory,} 2nd ed Springer, 
New York, (1998).

\bibitem[HaMi]{HaMi}
D. W. Hahn, R. Miranda, \emph{
Quadruple covers of algebraic varieties}, J. Algebraic Geom. 8 no. 1,(1999), 1-30.


\bibitem[Ha]{Ha}
R. Hartshorne, \emph{Algebraic geometry}, Grad. Texts Math., Vol. 52,
Springer Verlag, New York-Heidelberg-Berlin, (1977).


\bibitem[HMo]{HMo} J. Harris, I. Morrison, \emph{
Slopes of effective divisors on the moduli space of stable curves},
Invent. Math. 99, (1990), 321-355.

%\bibitem [HM] {HM} J. Harris, D. Mumford, 
%\emph{On the Kodaira dimension of the moduli space of curves},
%Invent. Math. 67, (1982), 23-86.

%\bibitem [HH]{HH}
%B. Hassett, D. Hyeon, \emph{ Log canonical models for 
%the moduli space of curves: the first divisorial contraction,} Trans. Amer. Math. Soc. 361, n. 8, (2009), 4471-4489.
%
\bibitem[K]{K} S. L. Kleiman, \emph{
Toward a Numerical Theory of Ampleness}, Annals of Mathematics, Second Series, Vol. 84 n. 3, (Nov., 1966), 293-344.



%\bibitem[K1]{K1} K. Konno, \emph{ A lower bound of the slope of
%trigonal fibrations}, Internat. J. Math. vol. {\bf 7}, n. 1, (1996), 19-27.
%
%\bibitem[K2]{K2} K. Konno,
%\emph{Clifford index and the slope of fibreed surfaces,}
%J. Algebraic Geom. 8, n. 2, (1999), 207-220.

\bibitem[L]{L} H. Laufer,
\emph{Normal two dimensional 
singularities}, Ann. of Math. studies 71, Princeton University 
Press, (1971). 

\bibitem[La]{La} R. Lazarsfeld, \emph{Positivity in algebraic geometry I}, Ergebnisse der Mathematik und ihrer Grenzgebiete. 3. Folge, Springer-Verlag, Berlin, (2004).

\bibitem[Mi]{Mi} Y. Miyaoka, \emph{The Chern classes and Kodaira dimension of a minimal variety}, Algebraic geometry, Sendai 1985, Adv. Stud. Pure Math. 10, North-Holland, Amsterdam, (1987), 449-476.

\bibitem[Mo]{Mo} A. Moriwaki, 
\emph{
A sharp slope inequality for general stable fibrations of curves}, 
J. Reine Angew. Math. 480, (1996), 177-195. 

%\bibitem[M]{M}
%D. Mumford,\emph{ Stability of Projective Varieties}, 
%L'Ens. Math. 23, (1977), 39-110.

\bibitem[Sch]{Sch} F. O. Schreyer, \emph{ Syzygies of 
canonical curves and special linear series}, Math. Ann. 275, n. 1, (1986), 105-137.

\bibitem[S]{S} Z. Stankova, \emph{ Moduli of trigonal curves, }
 J. Algebraic Geom. 9, n. 4, (2000), 607-662.
 

\bibitem[St]{St} L. Stoppino, \emph{Slope inequalities
for fibred surfaces via GIT}, Osaka J. Math. 45, n. 4, (2008), 1027-1041.

\bibitem[V1]{V1} 
 E. Viehweg, \emph{Die Additivit\"{a}t der Kodaira Dimension f\"{u}r projektive Faserr\"{a}ume \"{u}ber Variet\"{a}ten des allgemeinen Typs}, Journ. reine angew. Math. 330, (1982), 132-142. 
 
 \bibitem[V2]{V2} 
 E. Viehweg, \emph{Quasi-Projective Moduli for Polarized Manifolds}, Ergebnisse der Mathematik und ihrer Grenzgebiete, Series of Modern Surveys in Mathematics,
Springer (1995).
 
 
%\bibitem[X1]{X1} G. Xiao, \emph{Irregularity of surfaces with a linear
%pencil}, Duke Math. J. 55, n. 3, (1987), 597-602.

\bibitem[X]{X} G. Xiao, \emph{Fibred algebraic
 surfaces with low slope,} Math. Ann. 276, (1987), 449-466.

 
\end{thebibliography}
\end{document}